\documentclass[12pt]{amsart}

\usepackage{graphicx}
\usepackage{amsfonts, amssymb, amsmath,amsthm}
\usepackage{mathrsfs}
\usepackage{stmaryrd}
\usepackage{xfrac}
\usepackage{MnSymbol}
\usepackage{bbm}
\usepackage[sans]{dsfont}    
\usepackage{tikz}
\usepackage[all]{xy}

\usetikzlibrary{arrows}
\usepackage{hyperref}
\usepackage{caption}
\usepackage{float}
\usepackage{environ}
\usepackage{varwidth}

       \theoremstyle{plain}
\newtheorem{thm}{Theorem}[section]
\newtheorem{lemma}[thm]{Lemma}
\newtheorem{corollary}[thm]{Corollary}

\newtheorem{example}[thm]{Example}

 \newcommand{\be}{\begin{equation}} 
\newcommand{\ee}{\end{equation}}

\newcommand{\RSCvx}{\operatorname{\mathbf{SCvx}}}

\newcommand{\DN}{\Delta_{\Nat}}
\newcommand{\Om}{\mathbf{\Omega}}

\newcommand{\Ai}{\operatorname{A \! \downarrow \! \iota}}
\newcommand{\M}{\operatorname{\mathbf{Std}}}
\newcommand{\Std}{\operatorname{\mathbf{Std}}}
\newcommand{\Set}{\mathbf{Set}}

\newcommand{\A}{\mathcal{A}}

\newcommand{\D}{\mathcal{D}}
\newcommand{\E}{\mathcal{E}}

\newcommand{\F}{\mathcal{F}}
\newcommand{\G}{\mathcal{G}}

\newcommand{\bN}{\mathbb{N}}
\newcommand{\Nat}{\mathds{N}} 
\newcommand{\cN}{\mathbbmss{N}}
\newcommand{\twos}{\mathbf{2}}
\newcommand{\two}{\mathbbmss{2}}
\newcommand{\one}{\mathbf{1}}

\newcommand{\Rinf}{\mathbb{R}_{\infty}}
\newcommand{\SCvx}{\operatorname{\mathbf{SCvx}}}
\newcommand{\T}{\mathcal{P}}

\newcommand{\mS}{\mathbf{\Sigma}}

\def\@normalsize{\@setsize\normalsize{14.5pt}\xiipt\@xiipt
\abovedisplayskip 12\p@ plus3\p@ minus7\p@
\belowdisplayskip \abovedisplayskip
\abovedisplayshortskip  \z@ plus3\p@
\belowdisplayshortskip  6.5\p@ plus3.5\p@ minus3\p@
\let\@listi\@listI}
\title{Giry algebras for standard measurable spaces}

\author{Kirk Sturtz}

\address{Nong Han, Thailand}
 \email{kirksturtz@yandex.com}   
\begin{document}

\vspace{-20cm}

\begin{abstract}
The notion of ``super convex spaces'' generalizes the idea of convex spaces by replacing finite affine sums with countable affine sums.  Using this notion permits a very elegant approach for analysis of the Giry monad on standard measurable spaces and identifying the $\G$-algebras for that monad.  We use Isbell duality and restrict the adjunction $\mathbf{Spec} \dashv \mathcal{O}$ to a proper subcategory of super convex spaces and separated standard measurable spaces.

\smallskip
\noindent \textbf{Keywords.} Giry algebras, Giry monad, standard measurable spaces, super convex spaces.

\end{abstract}
 \maketitle
 \vspace{-.3in}
 
 \tableofcontents
 \addtocontents{toc}{ \vskip-60pt}
 
 \thispagestyle{empty}
 
  \section{Introduction}
  The category of convex spaces is well known and convex spaces are used throughout the literature.  Yet the notion of extending the characteristic property of finite affine sums to countably affine sums has not been used outside of a few category theory articles.  But the category of super convex spaces, $\SCvx$, provides an ideal setting for modeling probability monads because that category can be modeled as a functor category which is a subcategory of the category of presheaves $\Set^{\DN^{op}}$, where $\DN$ is the one object category (monoid) consisting of  the countably infinite-dimensional simplex and all the countably affine endomorphisms of that object which are easy to explicitly characterize because any countably affine map from $\DN$ to any other affine space is uniquely determined by where such a map sends the Dirac measures $\delta_i$.  Equivalently, the object $\DN$ can be viewed as consisting of all the elements $\mathbf{p} \in \G{\bN}$, where $(\G, \mu, \eta)$ is the Giry monad, and $\G{\bN}$ is the set of all probability measures on the set of natural numbers endowed with the discrete $\sigma$-algebra.  The full subcategory of $\SCvx$, consisting of the single object $\DN$, which is $\G{\bN}$ viewed as a super convex space,  is a dense subcategory of $\SCvx$, and hence the restricted Yoneda embedding permits us to use Isbell duality to relate the category $\SCvx \hookrightarrow \Set^{\Delta_{\Nat}^{op}}$  to its algebraic dual.

  We recall that for any small category $\Om$ Isbell duality gives an adjunction $\mathcal{O} \dashv \mathbf{Spec}$, 
\begin{figure}[H]
\begin{equation}  \nonumber
 \begin{tikzpicture}[baseline=(current bounding box.center)]

       \node  (St)  at  (3,0)    {$\mathcal{A}$};
       \node   (SC) at  (0,0)   {$\mathcal{X}$};
       \node   (Om)  at  (1.5, -1.5)  {$\Om$};
       \node   (SOm) at  (3,1.5)  {${(\Set^{\Om})}^{op}$};
       \node   (SOmop) at  (0, 1.5)  {$\Set^{\Om^{op}}$};
       \node    (c)  at  (6, 1.5)    {$\mathcal{O} \dashv \mathbf{Spec}$};
       \node     (c2) at   (8.3,.5)    {$\mathcal{O}(\F)[\omega] = \Set^{\Om^{op}}(\F,\Om(\cdot, \omega)) \quad \textrm{for all }\omega \in \Om$};
       \node    (c3)  at   (8.4, -.5)    {$\mathbf{Spec}(\G)[\omega] = \Set^{\Om}(\G,\Om(\omega, \cdot)) \quad \textrm{for all }\omega \in \Om$};
       
       \draw[->, above] ([yshift=2pt] SOmop.east) to node {$\mathcal{O}$} ([yshift=2pt] SOm.west);
       \draw[->,below]  ([yshift=-2pt] SOm.west) to node {$\mathbf{Spec}$} ([yshift=-2pt] SOmop.east);
       \draw[->,right] (St) to node {$\mathbf{y}^{op}$} (SOm);
       \draw[->,left] (SC) to node {$\mathbf{y}$} (SOmop);
       \draw[->,above] ([yshift=2pt] SC.east) to node {$\mathcal{G}$} ([yshift=2pt] St.west);
       \draw[->,below] ([yshift=-2pt] St.west) to node {$\F$} ([yshift=-2pt] SC.east);
       \draw[->,left] (Om) to node {$\iota$} (SC);
       \draw[->,right] (Om) to node {$j$} (St);
       
\end{tikzpicture}
\end{equation}
\caption{Isbell duality between the functor categories, $\Set^{\Om^{op}}$ (geometric spaces) and $\Set^{\Om}$ (algebraic spaces) where a canonical adjunction exists,  and its restriction to subcategories embedded into those functor categories.}
\label{Isbell}
\end{figure}
\noindent
and provided there is a dense functor $\iota: \Om \rightarrow \mathcal{X}$ and a codense functor $j: \Om \rightarrow \mathcal{A}$, then the two Yoneda embeddings are both full and faithful and we have the situation depicted in Diagram \ref{Isbell}.  

By choosing $\Om$ to be the full subcategory of $\SCvx$ consisting of the two objects $\cN$ and $\DN$, we have the situation in Diagram \ref{Isbell} with $\mathcal{X}$ a subcategory of $\SCvx$, and $\mathcal{A}$ a subcategory of  measurable spaces.  The super convex space $\cN$, which is the set of natural numbers with the only possible super convex space structure, is needed because the functor $j=\mS': \Om \rightarrow \Std_2$,  where $\Std_2$ is the category 
 of separated standard measurable spaces, is a codense functor defined by $\bN := (\Nat, \powerset{\Nat}) = \mS'(\cN)$ and $\mS'(\DN) = \G{\bN}$.    In other words, the functor $\mS'$ forgets the super convex space structure and endows the underlying sets with a $\sigma$-algebra structure allowing us to view those sets as measurable spaces.  

 The choice of $\mathcal{A} = \Std_2$ presents no difficulty to our main goal of characterizing the $\G$-algebras of the Giry monad because any measurable space $X$ which is not separated, implying there exists two points $x, y \in X$ such that $\delta_x=\delta_y$, cannot have any $\G$-algebra $h$ by the requirement that $h \circ \eta_X = \one_X$. 

 We can arrive naturally at the category $\Om$, as is carried out  in \S 2,  by recognizing the existence of the $\G$-algebra $\epsilon_{\bN}: \G{\bN} \rightarrow \bN$, which maps $\sum_{i \in \Nat} p_i \delta_i \mapsto min_i \{ i \, | \, p_i>0 \}$, and then noting that both the underlying sets of the two measurable spaces $\G{\bN}$ and $\bN$ have a super convex space structure, and we denote those two sets viewed as super convex spaces by $\DN$ and $\cN$, respectively.  
 
 \vspace{.1in}
 
To understand the big picture  let $\mS'': \Om \rightarrow \mathbf{Meas}$ denote the functor mapping the object $\cN \mapsto \bN$ and $\DN \mapsto \G{\bN}$.  We consider the functor $\mS''$ because $\mathbf{Meas}$ has all limits, and hence the right Kan extension $Ran_{\iota}(\mS'')$ exists and can be computed pointwise.  

 \begin{lemma}  The right Kan extension of the functor $\mS''$ along the inclusion function \mbox{$\iota: \Om \hookrightarrow \SCvx$},
\begin{equation} \nonumber
\begin{tikzpicture}
        \node  (Omega)  at  (0,0)  {$\Om$};
        \node  (SC)          at   (3,0)  {$\SCvx$};
        \node   (Std2)        at     (0, -1.5)  {$\mathbf{Meas}$};
        \draw[->,above] (Omega) to node {$\iota$} (SC);
        \draw[->,left] (Omega) to node {$\mS''$} (Std2);
        \draw[->,right,dashed] (SC) to node [yshift=-4pt]{$Ran_{\iota}(\mS'')$} (Std2);
\end{tikzpicture}
\end{equation}
has the universal arrow $\varepsilon:Ran_{\iota}(\mS'') \circ \iota \Rightarrow \mS''$ specified at component $\cN$ by $\varepsilon_{\cN} = \one_{\bN}$ and at component $\DN$ by $\varepsilon_{\DN} = \one_{\G{\bN}}$. (Hence $Ran_{\iota}(\mS'')$ really is an extension of $\mS''$.)
\end{lemma}
\begin{proof} We need to show that  the natural transformation $\varepsilon$ is a universal arrow, i.e., given any functor $\mathcal{H}: \SCvx \rightarrow \mathbf{Meas}$ and natural transformation $\alpha: \mathcal{H}\circ \iota \Rightarrow \mS''$ that there exists a unique natural transformation $\widehat{\alpha}: \mathcal{H} \rightarrow \widehat{\mS}$ such that the diagram on the right hand side
\begin{equation} \nonumber
\begin{tikzpicture}
        \node  (whS)  at  (0,0)  {$Ran_{\iota}(\mS'') \circ \iota$};
        \node  (Sp)          at   (3,0)  {$\mS''$};
        \node   (Hi)        at     (0, -1.5)  {$\mathcal{H} \circ \iota$};
        \node   (B)        at    (-3, 0)   {$Ran_{\iota}(\mS'')$};
        \node   (A)        at     (-3, -1.5)  {$\mathcal{H}$};
        \node   (c)        at    (-3.5, -2.5)  {in $Func(\SCvx, \mathbf{Meas})$};
        \node   (d)        at     (1, -2.5)   {in $Func(\Om, \mathbf{Meas})$};
        \draw[->,double,above] (whS) to node {$\varepsilon$} (Sp);
        \draw[->,double,below] (Hi) to node {$\alpha$} (Sp);
        \draw[->,double,left] (Hi) to node {$\widehat{\alpha} \circ \iota$} (whS);
        \draw[->,double,left] (A) to node {$\widehat{\alpha}$} (B);

\end{tikzpicture}
\end{equation}
commutes.  Given $\mathcal{H}$ and $\alpha$ we know the two measurable maps $\alpha_{\cN}: \mathcal{H}(\cN) \rightarrow \bN$ and $\alpha_{\Delta_{\bN}}: \mathcal{H}(\DN) \rightarrow \G{\bN}$ are the identity maps, hence $\mathcal{H}(\cN)=\bN$ and $\mathcal{H}(\DN)=\G{\bN}$.  To make the right hand side diagram commute we define $\widehat{\alpha}_{\cN} = \alpha_{\cN}$ and $\widehat{\alpha}_{\DN} = \alpha_{\DN}$.  

To see that this specification uniquely specifies $\mathcal{H}$ on all of $\SCvx$ let $A$ be an arbitrary object in $\SCvx$.  The functor $Ran_{\iota}(\mS'')$ on $A$ is determined by the 
\be \nonumber
Ran_{\iota}(\mS'')(A) = \lim \D = \lim \big( A \! \downarrow \! \iota \stackrel{\pi}{\longrightarrow} \Om \stackrel{\mS''}{\longrightarrow} \mathbf{Meas} \big),
\ee
and we can construct the limit in the usual fashion using the set of all cones over the diagram $\D$,
\be \nonumber
\begin{tikzpicture}
       \node  (Cone)  at  (0,0)   {$Cone(\one,\D)$};
       \node  (Df)        at   (2,1)   {$\D_f$};
       \node   (Dg)      at    (2, -1)  {$\D_g$};
       \draw[->,above] (Cone) to node {$\lambda_f$} (Df);
       \draw[->,below] (Cone) to node {$\lambda_g$} (Dg);
       \draw[->,right] (Df) to node {$\D_{f \stackrel{\phi}{\longrightarrow} g}$} (Dg);
\end{tikzpicture}
\ee
where for every component $f \in A \! \downarrow \! \iota$ the projection map $\lambda_f(\alpha) := \alpha_f(\star) \in \D_f$, and where each object $\D_f \in \{ \bN, \G{\bN} \}$.  One then endows $Cone(\one, \D)$ with the initial $\sigma$-algebra such that each projection map $\lambda_f$ is a measurable function.  Hence the 
\be \nonumber
\lim \D = \big( (Cone(\one,\D), \Sigma_{init}), \{\lambda_f \, | \, \forall f \in_{ob} A \! \downarrow \! \iota\}\big).
\ee

Given $A \in \SCvx$ and the functor $\mathcal{H}$ we have  $(\mathcal{H}(A), \{\mathcal{H}f: \mathcal{H}A \rightarrow cod(f) \, | \, f\in A\! \downarrow \iota\})$ specifies a cone over $\D$ where $cod(f) = \bN$ or $cod(f) = \G{\bN}$.  Consequently we obtain a measurable map  $\widehat{\alpha}_A: \mathcal{H}(A) \rightarrow \lim \D$ which is the unique arrow guaranteed by the universal property of $\lim \D = Ran_{\iota}(\mS'')(A)$. These unique maps, one for each object in $\SCvx$, specify the components of the natural transformation $\widehat{\alpha}$.

\end{proof}

Now take the pullback of $Ran_{\iota}(\mS'')$ along the inclusion functor $\Std_2 \hookrightarrow \mathbf{Meas}$ 
 \begin{equation}  \nonumber
\begin{tikzpicture}
    \node   (SCvxstar)  at  (0,0)   {$\SCvx_{\star}$};
    \node   (Std)    at   (4,0)   {$\Std_2$};
    \node   (SCvx)    at  (0,-1.5)  {$\SCvx$};
    \node   (Meas)     at  (4., -1.5)   {$\mathbf{Meas}$};
    
    \draw[->,below] (SCvx) to node {$Ran_{\iota}(\mS'')$} (Meas);
    \draw[>->,right] (Std) to node {$\iota$} (Meas);
    \draw[>->,left,dashed] (SCvxstar) to node {$$} (SCvx);
    \draw[->,above,dashed] (SCvxstar) to node {$\mathbf{\Sigma}$} (Std);

\end{tikzpicture}
\end{equation}
to define the functor $\mS: \SCvx_{\star} \rightarrow \Std_2$.  The reason we require that subcategory of $\SCvx$ is because there are many super convex spaces for which the right Kan extension $Ran_{\iota}(\mS'')(A)$ is not a standard measurable spaces.  For example, $Ran_{\iota}(\mS'') (\prod_{i \in [0,1]} \DN) = \prod_{i \in [0,1]}\G{\bN}$ which is not a standard measurable space. 

  The functor $\G$, viewed as the functor  $\T: \Std_2 \rightarrow \SCvx_{\star}$ is left adjoint to $\mS$.  The general argument is as follows. Suppose we are given a separated standard measurable space $X$ with the countable generating basis $\{U_i\}_{i \in \Nat}$ so that $\Sigma_X=\sigma(\mathcal{F})$ where  $\mathcal{F}: X \rightarrow \bN$ specifies the field which generates the $\sigma$-algebra on $X$.  In other words, $U_i = \mathcal{F}^{-1}(i)$.  Then $\T{X}$ is a super convex space which lies in $\SCvx_{\star}$ because if we compute the right Kan extension $Ran_{\iota}(\mS'')\big(\T{X})$ we obtain $\G{X}$ since the set of evaluation maps $ev_{U_i}: \T{X} \rightarrow \DN$ given by $ev_{U_i}(P)= P(U^c)\delta_0 + P(U) \delta_1$ are all countably affine maps and coseparate $\T{X}$.  Hence the right Kan extension of $\T{X}$ is just the underlying set of $\T{X}$ with the initial $\sigma$-algebra making the evaluation maps measurable functions - which is precisely how the Giry monad assigns the $\sigma$-algebra to the underlying set of $\G{X}$. This argument proves, for $\G$ the Giry monad on $\Std_2$, that  
  
\begin{lemma} \label{sigmaAlgebra} The composite functor $\mS \circ \T = \G$.
\end{lemma}
  
  Of course to prove that $\T \dashv \mS$ it is necessary, among other things,  to show that the counit of the proposed adjunction yields a barycenter map $\epsilon_A: \T(\mS A) \rightarrow A$ for every $A \in_{ob} \SCvx_{\star}$. This follows from the observation that since $\mS': \Om \rightarrow \Std_2$ is codense the measurable functions  $f:\mS A \rightarrow \bN$ characterize $\mS A$,
   \be \nonumber
 \lim \big( \operatorname{\mS A \! \downarrow \! \mS}' \rightarrow \Om \stackrel{\mS'}{\longrightarrow} \Std_2\big) =   (\mS A, \{f: \mS A \rightarrow \bN\}). 
  \ee
   But we can construct a family of  composite arrows $\epsilon_{\bN} \circ \G{f}: \G(\mS A) \rightarrow \bN$ specifying a cone over  the diagram \mbox{$\operatorname{\mS A \! \downarrow \! \mS}' \rightarrow \Om \stackrel{\mS'}{\longrightarrow} \Std_2$},
and hence there exists a unique measurable function $\G(\mS A) \rightarrow \mS A$. Further standard arguments show that unique arrow is in fact countably affine.   The full argument is carried out in \S \ref{sec:barycenter}.

\section{The fundamental $\G$-algebra}
 The set of all probability measures on the measurable space $\bN$ can be characterized as
 \be \nonumber
 \G{\bN} = \{ \sum_{i \in \Nat} p_i \, \delta_i \, | \, \textrm{ for all sequences }\mathbf{p}: \Nat \rightarrow [0,1] \textrm{ such that }  \lim_{N \rightarrow \infty} \{ \sum_{i=0}^{N} p_i \}= 1\},
 \ee
 where each $\delta_i$ is the Dirac measure on  $i \in \Nat$, and $\mathbf{p}(i)=p_i$.   Thus a probability measure on $\bN$ is equivalent to specifying a sequence $\mathbf{p}: \Nat \rightarrow [0,1]$ satisfying the condition $\lim_{N \rightarrow \infty} \{\sum_{i=0}^N p_i \} = 1$, and hence we  write  \mbox{$\mathbf{p} \in \G{\bN}$}. 
 
We claim the function
 \begin{equation}   \label{epsilonN}
 \begin{tikzpicture}[baseline=(current bounding box.center)]
          \node   (GN)   at   (0,0)    {$\G{\bN}$};
          \node    (N)     at    (3,0)   {$\bN$};
          \node     (p)    at    (0, -.8)   {$\displaystyle{\sum_{i \in \Nat}} p_i \, \delta_i$};
          \node     (q)  at    (3, -.6)   {$\displaystyle{min_{i}} \{i \, | \, p_i>0 \}$};
          \node      (ph)  at  (.5, -.6)   {};
          
          \draw[->,above] (GN) to node {$\epsilon_{\bN}$} (N);
          \draw[|->] (ph) to node {} (q);
    \end{tikzpicture}
\end{equation}
is a $\G$-algebra which is not free.  First note that

 \begin{lemma} \label{generalProp} The function  \mbox{$\epsilon_{\mathbb{N}}: \G{\mathbb{N}} \rightarrow \bN$}  is a  measurable map.
\end{lemma}
\begin{proof}
If $W \in \powerset{\Nat}$, by definition of the $\sigma$-algebra on $\G{\bN}$, each function $ev_W: \G{\bN} \rightarrow [0,1]$ given by $ev_W( \sum_{i\in \Nat} p_i \delta_i) = \sum_{i \in W} p_i$ is measurable. Taking  $W = \downarrow \! n = \{0,1,\ldots,n-1\}$, it follows that $ev_{ \downarrow \! n}^{-1}( 0 ) = \{ \sum_{i \in \Nat} p_i \delta_i  \in \G{\mathbb{N}} \, | \, p_i = 0 \textrm{ for all }i < n\}$  is a measurable set in $\G(\mathbb{N})$.
Since $\epsilon_{\bN}^{-1}(n) = ev_{ \downarrow \! n}^{-1}( 0 )$ we conclude the function $\epsilon_{\bN}$ is measurable. 
\end{proof}

 \begin{lemma} \label{Galgebra} The function  \mbox{$\epsilon_{\mathbb{N}}: \G{\mathbb{N}} \rightarrow \bN$}  is a non-free $\G$-algebra.
\end{lemma}
\begin{proof} 
 The function $\epsilon_{\bN}$  
  satisfies the unit law $\epsilon_{\bN} \circ \eta_{\bN} = \one_{\bN}$ where $\eta_{\bN}(j) = \delta_j$ is the unit of the Giry monad, and the associative law  $\epsilon_{\bN} \circ \G(\epsilon_{\bN}) =  \epsilon_{\bN} \circ \mu_{\bN}$ follows from analysis of the given expressions $\mu_{\bN}(Q)$ and $\G\epsilon_{\bN}(Q)$ on the measurable subset $\{k\} \in \powerset{\Nat}$.
We have
\be \nonumber
\mu_{\bN}(Q)(k)= \int_{p \in \G{\bN}} ev_k(p) \, dQ = \int_{p \in \G{\bN}}p_k \, d(Q )
\ee
and hence 
\be \label{emQ}
\begin{array}{lcl}
\epsilon_{\bN}( \mu_{\bN}(Q)) &=& \epsilon_{\bN}\bigg( \sum_{i=0}^{\infty} \big( \int_{p \in \G{\bN}} p_i \, dQ  \big)\,  \delta_i \bigg) \\
&=& min_{i} \{i \, | \, \int_{p \in \G{\bN}} p_i \, dQ  >0 \}
\end{array}.
\ee

On the other hand we have
\be \nonumber
\big(\G{\epsilon_{\bN}}(Q)\big)(k)=Q\big(\epsilon_{\bN}^{-1}(k)\big)=Q\big( \{\sum_{j=k}^{\infty} p_j \delta_j \, | \, \sum_{j=k}^{\infty} p_i = 1 \,  \textrm{ and } p_k>0\} \big)
\ee
and hence
\be \label{eGe}
\begin{array}{lcl}
\epsilon_{\bN}(\G{\epsilon_{\bN}}(Q)) = min_i \{i \, | \, Q(\{ \sum_{j \in \Nat} p_j \delta_j \, | \, \sum_{j\in \Nat} p_j =1 \, \textrm{ and }p_i>0 \}) >0 \}.
\end{array}
\ee
 A moments thought shows that the two expressions given in equations (\ref{emQ}) and (\ref{eGe}) are just two representations of the same quantity.  Hence we have shown the measurable function $\epsilon_{\bN}$ is a $\G$-algebra, 
 and by Lemma \ref{generalProp} the function $\epsilon_{\bN}$ is a measurable function. Hence $\epsilon_{\bN}$ is a $\G$-algebra.
It is clearly not a free algebra because free algebras are specified using the natural transformation $\mu$ of the Giry monad.
\end{proof}
 
 Similarly, for $n$ any finite integer the function
 \begin{equation}   \nonumber
 \begin{tikzpicture}[baseline=(current bounding box.center)]
          \node   (GN)   at   (0,0)    {$\G{\mathbbmss{n}}$};
          \node    (N)     at    (3,0)   {$(\mathbbmss{n}, \powerset{\mathbbmss{n}})$};
          \node     (p)    at    (0, -.8)   {$\displaystyle{\sum_{i \in \mathbbmss{n}}} p_i \, \delta_i$};
          \node     (q)  at    (3, -.6)   {$\displaystyle{min_{i \in \mathbbmss{n}}} \{i \, | \, p_i>0 \}$};
          \node      (ph)  at  (.5, -.6)   {};
          
          \draw[->,above] (GN) to node {$\epsilon_{\mathbbmss{n}}$} (N);
          \draw[|->] (ph) to node {} (q);
    \end{tikzpicture}
\end{equation}
is a $\G$-algebra which is not free.

\vspace{.1in}
These $\G$-algebras are interesting because we can characterize them as a morphism in the category of super convex spaces, $\SCvx$.  This is not surprising because the objects of the category $\SCvx$ are themselves defined in terms of the elements of the underlying set of $\G{\bN}$.  

 Given any set $A$, a sequence $\mathbf{a}: \Nat \rightarrow A$, and any   $\mathbf{p} \in \G{\bN}$   we refer to the formal expression  ``$\sum_{i \in \Nat} p_i \, a_i$''  as a \emph{countably affine sum} of elements of $A$, and for brevity we use the notation ``$\sum_{i\in \Nat} p_i a_i$''  to refer to a countably affine sum dropping the explicit reference to the condition that the limit of partial sums $\sum_{i=0}^N p_i$ converges to one.  An alternative notation to the countable affine sum notation is to use the integral notation
 \be \nonumber
  \int_{\Nat} \mathbf{a} \, d\mathbf{p} := \sum_{i \in \Nat} p_i a_i.
 \ee  
We say a set $A$ has the structure of a super convex space  if it comes equipped with a function
\be \nonumber
\begin{array}{ccccc}
st_A&:& \G{\bN} \times \Set(\Nat, A) & \rightarrow& A \\
&:& (\mathbf{p}, \mathbf{a}) & \mapsto & \int_{\Nat} \mathbf{a} \, d\mathbf{p}
\end{array}
\ee
such that the following two axioms are satisfied:
\begin{quote}

\textbf{Axiom 1}. For every sequence $\mathbf{a}: \Nat \rightarrow A$ and every  $j \in \Nat$ the property 
\be \nonumber
\int_{\Nat} \mathbf{a} \, \, d\delta_j = a_j
\ee
holds.  

\vspace{.075in} 

\noindent
\textbf{Axiom 2}. If $\mathbf{p} \in \G{\bN}$ and $\mathbf{Q}: \Nat \rightarrow \G{\bN}$ is a sequence of probability measures on $\bN$   then 
\be \nonumber
\int_{j \in \Nat} \big( \int_{\Nat} \mathbf{a} \, d\mathbf{Q}^j \big) \, d\mathbf{p} = \int_{\Nat} \mathbf{a} \,  d(  \int_{j \in \Nat} \mathbf{Q}^{\bullet} \, d\mathbf{p}  \big).
\ee
Stated alternatively,  $ \displaystyle{\sum_{j \in \Nat}} p_j ( \displaystyle{\sum_{i \in \Nat}} Q^j_i a_i) 
= \displaystyle{\sum_{i \in \Nat}}( \displaystyle{\sum_{j \in \Nat}} p_j Q_i^j)a_i$.

\end{quote}
The second axiom uses  the pushforward measure $\G(\mathbf{Q})\mathbf{p} \in \G^2{\bN}$ and the natural transformation $\mu$ of the Giry monad at component $\bN$, \mbox{$\mu_{\bN}: \G^2{\bN} \rightarrow \G{\bN}$}, which yields the probability measure on  the measurable space $\bN$  whose value at the measurable set $\{j\}$ is given by the composite of measurable maps
\be \nonumber
\begin{tikzpicture}
      \node   (1)  at  (0,0)  {$\one$};
      \node   (GN) at  (-3,1.5)  {$\G{\bN}$};
      \node   (G2N) at  (0, 1.5)  {$\G^2{\bN}$};
      \node   (GN2)  at   (2.5,1.5)  {$\G{\bN}$};
      \node    (01)   at   (5,1.5)   {$[0,1]$};
      \draw[->,above] (GN2) to node {$ev_{\{j\}}$} (01);
      \draw[->,above] (G2N) to node {$\mu_{\bN}$} (GN2);
      \draw[->,left] (1) to node [xshift=-2pt,yshift=-2pt]{$\mathbf{p}$} (GN);
      \draw[->,above] (GN) to node {$\G{\mathbf{Q}}$} (G2N);
      \draw[->,left] (1) to node [xshift=-2pt,yshift=3pt]{\small{$ \int_{j \in \Nat} \mathbf{Q}^{j} \, d \mathbf{p}$}} (GN2);
      \draw[->,right] (1) to node [xshift=0pt,yshift=-17pt]{$\begin{array}{lcl}
\mu_{\bN}\big( \G(\mathbf{Q})\mathbf{p}\big)(\{j\}) &=& \int_{\{j\} }d\big(\G(\mathbf{Q})\mathbf{p}\big)  \\
&=& \int_{ \Nat} \mathbf{Q}^j \, d\mathbf{p} \\
&=& \sum_{i \in \Nat} p_i Q^j_i 
\end{array}$} (01);
\end{tikzpicture}
\ee
By taking $\mathbf{Q}^{\bullet}:\Nat \rightarrow \G{\bN}$ to be the constant sequence with value $\delta_j$ we see that the first axiom follows from the second axiom.

\vspace{.3in}

A morphism of super convex spaces, called a countably affine map, is a set function $m: A \rightarrow B$ such that
\be \nonumber
m\big( \int_{\Nat} \mathbf{a} \, d\mathbf{p}\big) = \int_{\Nat} (m \circ \mathbf{a}) \, d\mathbf{p},
\ee 
where the composite  $m \circ \mathbf{a}$ gives the sequence in $B$  with component $m(a_i)$.
Composition of countably affine maps is the set-theoretical composition.
Super convex spaces with morphisms the countably affine maps form a category denoted $\SCvx$. 

\vspace{.1in}

Having formally defined the category of super convex spaces we now note that the set of natural numbers $\Nat$  has a super convex space structure defined, for all $\mathbf{p} \in \G{\bN}$,  by 
  \be \nonumber
\displaystyle{\sum_{i \in \Nat}} p_i  \, \underline{i} = min_i \{ i \, | \, \textrm{such that }p_i>0 \},
\ee
and  the set $\Nat$ with this super convex space structure is denoted $\cN$ to distinguish it from the measurable space $\bN$.

\begin{lemma} The measurable function $\epsilon_{\bN}$ defined in equation (\ref{epsilonN}) is a countably affine map.
\end{lemma}
\begin{proof}

The super convex space structure of $\cN$ makes the measurable function $\epsilon_{\bN}: \DN \rightarrow \cN$ a countably affine map because 
\be \nonumber
\epsilon_{\bN}(\sum_{i \in \Nat} p_i \, \delta_i)= \sum_{i \in \Nat} p_i \, \underline{i} =  \min_i \{i \, | \, p_i>0\},
\ee
where the underline notation ``$\underline{i}$''  is used to distinguish the coefficients ``$p_i$'' from  the elements of $\Nat$ because the coefficients can assume the values $0$ and $1$. 
\end{proof}

\section{Basic properties of super convex spaces}    \label{defSCvx}
The most fundamental property of super convex spaces is
\begin{lemma} \label{ad} For $A$  any super convex space  every countably affine map $m \in \SCvx(\DN, A)$ is uniquely specified by a sequence in $A$, hence we have $\SCvx(\DN, A) \cong  \Set(\Nat,A)$.
\end{lemma}
\begin{proof}
Every element $\mathbf{p} \in \DN$ has a unique representation as a countable affine sum $\mathbf{p} = \sum_{i \in \bN} p_i \delta_i$, and hence a countably affine map $m:\DN \rightarrow A$ is uniquely determined by where it maps each Dirac measure $\delta_i$.  Thus  $i \mapsto m(\delta_i)$  specifies a sequence in $A$.  
\end{proof}

Given any sequence of elements $\mathbf{a}: \Nat \rightarrow A$ we denote the countably affine map $\DN \rightarrow A$ specified by $\delta_i \rightarrow a_i$ with the notation $\langle \mathbf{a} \rangle$ or just $\langle a_i \rangle$ with the brackets implying we are viewing it as a countably affine map from $\DN$.   A sequence of elements $\mathbf{Q}: \Nat \rightarrow \DN$ specifies a countably affine map $\langle \mathbf{Q} \rangle: \DN \rightarrow \DN$ which can be interpreted as a countably affine transformation of the spanning set of elements, mapping  $\delta_i \mapsto Q_i$ where $Q_i = \sum_{j \in \Nat} q^i_j \delta_j$, and the transformation amounts to
 \begin{equation}   \nonumber
 \begin{tikzpicture}[baseline=(current bounding box.center)]
      \node   (DN) at  (0,0)  {$\DN$};
      \node   (DN2) at  (0, -1.6)  {$\DN$};
      \node     (A)   at    (2, -.8)   {$A$};
      \draw[->,above] (DN) to node {$\langle \mathbf{a} \rangle$} (A);
      \draw[->,below] (DN2) to node {$\langle \mathbf{b} \rangle$} (A);
      \draw[->,left] (DN2) to node {$\langle \mathbf{Q} \rangle$} (DN);
      \node    (c)  at   (4.5, -.8)   {$b_i = \displaystyle{\sum_{j \in \Nat}} q^i_j a_j .$};
   \end{tikzpicture}
\end{equation}

When the sequence $\mathbf{a}: \Nat \rightarrow A$ is itself a countably affine map $\mathbf{a}: \Nat \rightarrow A$ then the $\SCvx$-diagram
  \begin{equation}   \nonumber
 \begin{tikzpicture}[baseline=(current bounding box.center)]
      \node   (DN) at  (0,0)  {$\DN$};
      \node    (N)    at   (0, -1.6)   {$\cN$};
      \node     (A)   at    (2, -.8)   {$A$};
      \draw[->,above] (DN) to node {$\langle \mathbf{a} \rangle$} (A);
      \draw[->,below] (N) to node {$\mathbf{a}$} (A);
      \draw[->,left] (DN) to node {$\epsilon_{\bN}$} (N);
   \end{tikzpicture}
\end{equation}
commutes.   

Countably affine endomaps on the super convex space $\cN$ have a simple characterization.

\begin{lemma} \label{NS} A function $f: \cN \rightarrow \cN$ is a countably affine map if and only if $f$ is monotone, $i < j$ implies $f(i) \le f(j)$.
\end{lemma}
\begin{proof}
\underline{Necessary condition}  Suppose that $f: \cN \rightarrow \cN$ is a countably affine map.  Let \mbox{$i < j$}. By the super convex space structure on $\bN$ it follows, for all $\alpha \in (0,1)$, that  $\alpha i + (1-\alpha) j = i$ .  If $f$ is not monotone then there exist a pair of elements \mbox{$i,j \in \Nat$} such that \mbox{$i <j$} with \mbox{$f(j) < f(i)$}.  This implies, for all $\alpha \in (0,1)$, that  \mbox{$f(j)= \alpha f(i) + (1-\alpha) f(j) <   f(\alpha i + (1-\alpha) j ) =f(i)$}, which contradicts our hypothesis that $f$ is a countably affine map.
 
 \vspace{.1in}
 
\underline{Sufficient condition} Suppose $f$ is a monotone function, and
 that we are given an arbitrary countably affine sum $\sum_{i \in \Nat} p_i i = n$ in $\cN$, so that for all $i=0,1,\ldots,n-1$ we have $p_i=0$.  Since the condition defining the super convex structure is conditioned on the property ``for $p_i \ne 0$'', the countably affine sum is not changed  by removing any number of terms $i$ in the countable sum whose coefficient $p_i=0$.  Hence 
for all $j$ such that $n< j$ it follows that $f(n) \le f(j)$ so that
\be \nonumber
f( \sum_{i=0}^{\infty} p_i \, i) = f(n) = \sum_{i=n}^{\infty} p_i f(i) 
\ee
where the last equality follows from the definition of the super convex space structure on $\mathbb{N}$.
\end{proof}

\begin{corollary}  \label{rightIdempotent} Any $m \in \SCvx(\DN, \cN)$ can be written uniquely as $m=\phi \circ \epsilon_{\bN}$ where $\phi \in \SCvx(\cN,\cN)$.
\end{corollary}
\begin{proof}  
By Lemma \ref{ad} every element $m \in \SCvx(\DN,\cN)$ can be represented by a map \mbox{$\langle \mathbf{u} \rangle: \DN \rightarrow \cN$} which is determined by the sequence  \mbox{$\mathbf{u}: \cN \rightarrow \cN$}, which maps $i \mapsto u_i$. The map $\epsilon_{\bN}(\delta_i) = i$ acts as an idempotent operator on the right since $\mathbf{u}(\epsilon_{\bN}(\delta_i))= u_i$.  Therefore  $m= \mathbf{u} \circ \epsilon_{\bN}$.
\end{proof}

In particular, this implies that given any countably affine map $\langle \mathbf{Q} \rangle: \DN \rightarrow \DN$ that there  exists a  countably affine map $\phi: \cN \rightarrow \cN$ such that the diagram
 \begin{equation}   \label{basicTriangle}
 \begin{tikzpicture}[baseline=(current bounding box.center)]
      \node   (DN) at  (0,0)  {$\DN$};
      \node   (DN2) at  (0, -1.6)  {$\DN$};

      \draw[->,left] (DN2) to node {$\langle \mathbf{Q} \rangle$} (DN);
       \node   (N2)  at (3, -1.6)  {$\cN$};
      \node    (N)  at   (3, -0)   {$\cN$};
      \draw[->, above] (DN) to node {$\epsilon$} (N);
      \draw[->,below] (DN2) to node {$\epsilon$} (N2);
      \draw[->,right] (N2) to node {$\phi$} (N);
   \end{tikzpicture}
\end{equation}
commute.

\begin{corollary} \label{phiCommutes} For every countably affine map $\phi: \cN \rightarrow \cN$ the property $\phi \circ \epsilon_{\bN} = \epsilon_{\bN} \circ \T{(\mS \phi)}$ holds.
\end{corollary}
\begin{proof}  A direct verification gives
\be \nonumber
\begin{array}{lcl}
\big(\phi \circ \epsilon_{\bN}\big)( \sum_{i \in \Nat} p_i \delta_i) &=& \phi\big( \sum_{i \in \Nat} p_i \underline{i} \big) \\
&=& \phi\big( \min_i \{ i \, | \, p_i>0 \} \big)
\end{array}
\ee
Suppose  $\min_i \{ i \, | \, p_i>0 \} = k$ which  implies $p_i = 0$ for all $i=0,1,\ldots, k-1$ and $p_k>0$. Hence $\big(\phi \circ \epsilon_{\bN}\big)( \sum_{i \in \Nat} p_i \delta_i) =\phi(k)$.  

On the other hand we have
\be \nonumber
\begin{array}{lcll}
\big(\epsilon_{\bN} \circ \T(\mS \phi) \big)( \sum_{i \in \Nat} p_i \, \delta_i) &=& \epsilon_{\bN}\big( \sum_{i \in \Nat} p_i \delta_{\phi(i)} \big) &\textrm{by def. of push forward map}\\
&=& \sum_{i \in \Nat} p_i \phi(i)  &\textrm{by def. of }\epsilon_{\bN}\\
&=& \min_i \{\phi(i) \, | \, p_i >0 \} & \textrm{using the $\SCvx$ structure on }\cN
\end{array}
\ee
Since $p_i=0$ for all $i=0,1,2,\ldots,k-1$ and $p_k>0$ the minimum is given by $\phi(k)$, which proves the lemma.

\end{proof}

We can classify a super convex space into one of three types.
\begin{enumerate}
\item A space that can be embedded into a vector space is of \emph{geometric type}.  
\item 
Let $\Delta_{\mathbb{N}}^{\dagger} = \{\mathbf{p} \in (0,1)^{\Nat} \, | \, \sum_{i \in \Nat}p_i = 1\}$.  A space $A$, consisting of at least two points, is  of \emph{discrete type} if, for each fixed sequence $\mathbf{a}:\mathbb{N} \rightarrow A$, the function
 \be \nonumber
 \begin{array}{ccc}
\Delta_{\mathbb{N}}^{\dagger} & \longrightarrow & A \\
\mathbf{p} & \mapsto & \sum_{i \in \Nat} p_i a_i
\end{array}
\ee
is a constant function.
\item A space which is neither of discrete type or geometric type is of \emph{mixed type}.
\end{enumerate}

\begin{lemma} \label{DG}  If  $D$ a discrete type space and $G$ a geometric type space the set $\SCvx(D,G)$ consist of constant maps. 
\end{lemma}
\begin{proof}
Since $D$ is a discrete space, it follows that for all $d_1,d_2 \in D$ and all $r \in (0,1)$, that the property $r d_1 + (1-r) d_2 \in \{d_1,d_2\}$  holds, which implies that  the condition $m( r d_1 + (1-r) d_2 ) = r m(d_1) + (1-r) m(d_2)$
can only be satisfied if $m$ is a constant function.
\end{proof}

Lemma \ref{DG} shows that no useful information can be obtained about a geometric space $G$ by looking at maps from a discrete type space $D$ into $G$.  On the other hand, the maps $\SCvx(G, D)$ can provide quite useful information. 

\begin{example}  \label{two}
The set $\two$ with the super convex structure defined by 
\be \nonumber
(1-r) \, 0 + r \, 1 = \left\{ \begin{array}{ll} 0 & \textrm{ if and only if }r \in [0,1) \\ 1 & \textrm{ otherwise} \end{array} \right.
\ee
is of discrete type.  The inclusion map $\two \hookrightarrow \bN$ is a countably affine map.

There is also a super convex space $\twos$ with the structure defined by 
\be \nonumber
(1-r) \, 0 + r \, 1 = \left\{ \begin{array}{ll}1 & \textrm{ if and only if }r \in (0,1] \\ 0 & \textrm{ otherwise} \end{array} \right.
\ee
and there is a $\SCvx$-isomorphism 
 $sw:  \twos \rightarrow \two$ defined by $sw(0)=1$ and $sw(1)=0$.  The function $\epsilon_{\twos}: \T{\twos} \rightarrow \twos$ given by $(1-p) \delta_0 + p \delta_1\mapsto 1$ if and only if $p \in (0,1]$ specifies a countably affine map which becomes a $\G$-algebra map $\epsilon_{\twos}:\G{\twos} \rightarrow (\twos, \powerset{\twos})$ using the obvious $\sigma$-algebra structures associated with those super convex spaces.
  
When we employ characteristic functions and want to view $\chi_U$ as a measurable map into $\bN$ it is necessary to use the swap isomorphism to obtain
 \be \nonumber
 X \xrightarrow{\chi_U} \twos \xrightarrow{sw_2} \two \hookrightarrow \bN,
 \ee
 and similarly,
  \be \nonumber
 \G(X) \xrightarrow{ev_U} \G(\twos) \xrightarrow{\G(sw_2)} \G(\two) \xrightarrow{\epsilon_{\two}} \two \hookrightarrow \bN.
 \ee
 where $\epsilon_{\two}: \G(\two) \rightarrow \two$ is defined by
 \be \nonumber
 \epsilon_{\two}( (1-r) \delta_0 + r \delta_1) = \left\{ \begin{array}{ll} 0 & \textrm{for all }r\in [0,1) \\ 1& \textrm{if and only if }r=1 \end{array} \right.
 \ee
 or, as is usually done, by first using the isomorphism $\G(\two) \cong [0,1]$ mapping $(1-r) \delta_0 + r \delta_1 \mapsto r$ and then defining $\epsilon_{\two}(r) = 0$ for all $r \in [0,1)$ and $\epsilon_{\two}(r)=1$ if and only if $r=1$.
 
\vspace{.1in}

For every finite number $n \in \Nat$, let $\mathbbmss{n}=\{0,1,\ldots,n-1\}$  endowed with the super convex space structure $\sum_{i \in \Nat} p_i u_i = min_{i} \{u_i \, |  u_i \in \mathbbmss{n} \, \textrm{ s.t. } p_i>0 \}$.  The inclusion map $\mathbbmss{n} \hookrightarrow \cN$ is a countably affine map.  Just as for the space $\two$, one can choose the max function instead to obtain the super convex space $\mathbf{n}$, and there is a swap map $sw_n: \mathbbmss{n} \rightarrow \mathbf{n}$ defined by $i \mapsto (n-1) - i$ which is an isomorphism.  It is only for the case $\cN$ that we must choose the minimum function since we need to employ the well-ordered property.

\end{example}

\begin{example}  \label{01Iso} The $n-1$ dimensional affine simplex \mbox{$\Delta_{n} = \{ \sum_{i=0}^{n-1} p_i \delta_i \, | \, \sum_{i=0}^{n-1} p_i =1 \ \textrm{and }p_i \in [0,1]\}$} is the space of all probability measures on the set of $n$ points.  It is a space of geometric type since it can be embedded into the vector space ${\mathbb{R}}^{n-1}$.  The $1$-dimensional simplex $\Delta_{2} \cong [0,1]$ via the $\SCvx$-isomorphism $(1-r) \delta_0 + r \delta_1 \mapsto r$.  

 It is straight forward to show  the function
\be \nonumber
\begin{array}{ccc}
[0,1] & \longrightarrow & \Delta_2  \\
r & \mapsto &  r \delta_0 + (1-r) \delta_1
\end{array}
\ee
is both a $\SCvx$-isomorphism and a $\mathbf{Meas}$-isomorphism when $[0,1]$ has the usual $\sigma$-algebra generated by the intervals.
\end{example}


\begin{example} \label{mixedType}
A space of mixed type is given by $\Rinf$ which is the 
real line $\mathbb{R}$ with one point adjoined, denoted $\infty$, which satisfies the property that any countably affine sum $\sum_{i \in \mathbb{N}} p_i r_i = \infty$ if either (1) $r_j = \infty$ and $p_j>0$ for any index $j$, or (2) the sequence of partial sums does not converge.  It is for the latter reason that $\mathbb{R}$ is not a super convex space since we could take $p_i = \frac{1}{2^i}$ and $r_i = 2^{i+1}$ and the limit of the sequence does not exist in $\mathbb{R}$.  Thus while $\mathbb{R}$ is a convex space it is not a super convex space.  

 The only nonconstant countably affine map $j: \Rinf \rightarrow \two$ is given by $j(u)=1$ for all $u \in \mathbb{R}$ and $j(\infty)=0$ (for the super convex space structure on $\two$ determined by $\frac{1}{2} \underline{0} + \frac{1}{2} \underline{1} = \underline{0}$).  
\end{example}

\begin{example}
Given  the inclusion  map $\iota: [0,1] \hookrightarrow \Rinf$ and the countably affine function $f: [0,1] \rightarrow \Rinf$ given by $x \mapsto \frac{1}{2}x + \frac{1}{2}$  the coequalizer is $j:\Rinf \rightarrow \two$ specified in Example (\ref{mixedType})  because the equivalence class of any $u \in \mathbb{R}$ is all of $\mathbb{R}$, and hence the coequalizer is given by $j$.   For discrete type spaces the effect is less dramatic.  For example, choose two monotonic functions $f,g: \bN \rightarrow \bN$ which disagree on a finite number of terms. (See Lemma \ref{NS}.)
\end{example}

\section{The dense functor $\iota: \Om \rightarrow \SCvx$}

 \begin{lemma} \label{epi}
 In $\SCvx$ the countably affine map $\epsilon_{\bN}$ is an epimorphism.
 \end{lemma}  
 \begin{proof}
 Suppose that $f,g: \cN \rightarrow A$ are any two parallel arrows in $\SCvx$, and that $f \circ \epsilon_{\bN} = g \circ \epsilon_{\bN}$.
 The countably affine maps are completely specified by where they send the elements of $\cN$, and hence  $(f \circ \epsilon_{\bN})(\delta_i) = (g \circ \epsilon_{\bN})(\delta_i)$ for all $i \in \Nat$, which says $f(i)=g(i)$ for all $i \in \cN$. Hence $f=g$.
 \end{proof}

\begin{lemma} \label{DNDense}
The full subcategory of $\SCvx$ consisting of the single object $\DN$ is  a dense subcategory of $\SCvx$.
\end{lemma}
\begin{proof}
 We show the  functor $\RSCvx \rightarrow \Set^{\DN^{op}}$ given by $A \mapsto \SCvx(\bullet, A)$ is full and faithful.
 
 \underline{Faithful} 
Let $\overline{a}: \DN \rightarrow A$ given by $\delta_i = a$ for all $i$  denote the constant maps with value $a$. 
Suppose that $f,g : A \rightarrow B$ with $f(a) \ne g(a)$.  Then the induced natural transformations $f_{\star},g_{\star}: \SCvx( \bullet,A) \rightarrow \SCvx(\bullet,B)$ specified by 
 composition on the left with $f$ and $g$, respectively, yield two distinct natural transformations since  $f_{\star}( \overline{a} ) = f(a) \ne g(a) = g_{\star}( \overline{a} )$.
 
\underline{Fullness}  Suppose that $J \in Nat(\SCvx(\bullet,A), \SCvx(\bullet,B))$.  Then at the component $\DN$, for every $a \in A$, we have $J(\overline{a})  \in  \SCvx(\DN,B)$ which is a constant map into $B$, hence specifies a unique point in $B$.  
Thus the function $J_{\DN}$ determines a set mapping $A \rightarrow B$ specified by $a \mapsto J_{\DN}(\overline{a})$.  

The fact that $J_{\DN}$ specifies a countably affine map follows from naturality.  We have the commutative $\SCvx$-diagram

 \begin{equation}   \nonumber
 \begin{tikzpicture}[baseline=(current bounding box.center)]
          \node   (D1)   at   (-.5,0)    {$\DN$};
          \node    (DN)    at     (2.5, 0)    {$\DN$};
          \node     (A) at  (1.,-1.5)  {$A$};

          \draw[->,above] (D1) to node {$\overline{\mathbf{p}}$} (DN);
          \draw[->,below, right] (DN) to node [yshift=-3pt] {\small{$\langle a_i \rangle$}} (A);
          \draw[->,below,left] (D1) to node {\small{$\sum_{i \in \mathbb{N}} p_i a_i$}} (A);
         
     \end{tikzpicture},
 \end{equation}
where $\overline{\mathbf{p}}$ is a constant map into the element $\mathbf{p} \in \DN$, and the composite map $\langle a_i \rangle \circ \overline{\mathbf{p}}$ is  the constant map $\sum_{i \in \mathbb{N}} p_i a_i: \DN \rightarrow \one \rightarrow A$  with value $\sum_{i \in \mathbb{N}} p_i a_i \in A$.
By naturality we have the commutative $\Set$-diagram
 \begin{equation}   \nonumber
 \begin{tikzpicture}[baseline=(current bounding box.center)]
          \node   (SCNA)   at   (0,0)    {$\SCvx(\DN,A)$};
          \node    (SCNB)    at     (5, 0)    {$\SCvx(\DN,B)$};
          \node     (SC1A)   at   (0,-1.5)    {$\SCvx(\DN,A)$};
          \node    (SC1B)    at    (5,-1.5)  {$\SCvx(\DN,B)$};
          
          \draw[->,above] (SCNA) to node {$J_{\DN}$} (SCNB);
          \draw[->,below] (SC1A) to node   {$J_{\DN}$} (SC1B);
          \draw[->,left] (SCNA) to node {$\SCvx(\overline{\mathbf{p}},A)$} (SC1A);
          \draw[->,right] (SCNB) to node {$\SCvx(\overline{\mathbf{p}},B)$} (SC1B);
         
         \node  (D1)    at   (-3,0)    {$\DN$};
         \node   (DN)  at    (-3,-1.5)  {$\DN$};
         \draw[->,left] (DN) to node {$\overline{\mathbf{p}}$} (D1);
         
          \node   (ain)   at   (0,-3)    {$\langle a_i \rangle$};
          \node    (aib)    at     (6.6, -3)    {$J_{\DN}( \langle a_i \rangle )$};
          \node     (alphai)   at   (0,-4.5)    {$\sum_{i \in \Nat} p_i a_i$};
          \node    (Galphai)    at    (3.7,-4.5)  {$J_{\DN}(\sum_{i \in \Nat} p_i a_i)$};
          \node    (Galphab)   at    (6.5,-4.5) {$=J_{\DN}( \langle a_i\rangle) \circ \overline{\mathbf{p}}$};
          \node    (pha)     at    (6.5, -3.3)  {};
          \node    (phb)    at     (6.5,-4.2)   {};
          
          \draw[|->] (ain) to node {} (aib);
          \draw[|->] (ain) to node   {} (alphai);
          \draw[|->](alphai) to node {} (Galphai);
          \draw[|->] (pha) to node {} (phb);

     \end{tikzpicture}.
 \end{equation}
\noindent
Since $J_{\DN}( \langle a_i\rangle) \in \SCvx(\DN, B)$ it is specified by a family of points, $J_{\DN}( \langle a_i \rangle)  = \langle b_i \rangle$.  The equality in the lower right hand corner thus shows that the map defined by $J_{\DN}$ on the constant functions $\overline{a}$ specifies a countably affine $A \rightarrow B$, and hence $\mathbf{y}$ is full, i.e., $J_{\DN} = \mathbf{y}(m)=\SCvx(\bullet,m)$ for some $m \in \SCvx(A,B)$. 
\end{proof}

 \begin{corollary} \label{dense}  The inclusion functor $\iota: \Om \rightarrow \SCvx$ is dense.  Equivalently, the restricted dual Yoneda mapping $\mathcal{Y}^{op}: \SCvx \rightarrow \Set^{\Om^{op}}$ is  full and faithful.
  \end{corollary}
 \begin{proof} 
 Suppose $J \in Nat( \SCvx(\iota, A), \SCvx(\iota,B))$.  
 By naturality  we have the commutative $\Set$-diagram
 
     \begin{equation}   \nonumber
 \begin{tikzpicture}[baseline=(current bounding box.center)]
          \node   (GNX) at  (-1.5,0)  {$\SCvx(\DN,A)$};
          \node   (GN)   at   (2.5,0)    {$\SCvx(\DN,B)$};        
          \node    (XN)  at    (-1.5, -1.5)  {$\SCvx(\cN,A)$};
          \node    (N)    at   (2.5, -1.5)    {$\SCvx(\cN,B)$};
          
          \draw[->,above] (GNX) to node {$J_{\DN}$} (GN);
          \draw[->,right] (N) to node {\small{$\SCvx(\epsilon_{\mathbb{N}},B)$}} (GN);
          \draw[->,left] (XN) to node [xshift=0pt,yshift=0pt] {\small{$\SCvx(\epsilon_{\mathbb{N}},A)$}} (GNX);
          \draw[->,below] (XN) to node {$J_{\cN}$} (N);
          
           \node   (U) at  (5.5,0)  {$f \circ \epsilon_{\bN}$};
          \node   (JU)   at   (9,0)    {$J_{\DN}(f \circ \epsilon_{\bN})=J_{\cN}(f) \circ \epsilon_{\bN}$};        
          \node    (k)  at    (5.5, -1.5)  {$f$};
          \node    (Jk)    at   (9.5, -1.5)    {\small{$J_{\cN}(f)$}};
          \node   (ph)   at    (9.5,-.3)   {};
          
          \draw[|->] (U) to node {} (JU);
          \draw[|->] (Jk) to node {} (ph);
          \draw[|->] (k) to node {} (U);
          \draw[|->] (k) to node {} (Jk);

     \end{tikzpicture},    
 \end{equation}
and by Lemma  \ref{DNDense} it follows that $J_{\DN} = \SCvx(\DN,m)$ for some $m \in \SCvx(A,B)$.  Hence the upper right hand corner in the diagram
is given by $m \circ f \circ \epsilon_{\bN} = J_{\cN}(f) \circ \epsilon_{\bN}$.  Since, by Lemma \ref{epi},  $\epsilon_{\bN}$ is an epimorphism it follows that $J_{\cN}(f) = m \circ f$, and hence $J_{\cN} = \SCvx(\cN,m)$. 
 \end{proof}

\section{Standard measurable spaces }  \label{sec:Std}

The two defining  characteristics of a standard measurable space $X$ are
\begin{enumerate}
\item  Its  $\sigma$-algebra $\Sigma_X$ is asymptotically generated by a countably generated field $\F$, so that $\Sigma_X = \sigma(\F)$.  
Let $\{F_i\}_{i=1}^{\infty}$ be the sequence of finite fields which asymptotically generate the field \mbox{$\F = \lim_{n \rightarrow \infty} \bigcup_{i=1}^n F_i$}.  

Each such field $F_n$ itself is generated by a partition of $X$, and we take $F_n$ as the field generated by the partition $Atoms_n: X \rightarrow \mathbf{n}$ which specifies the atoms of the finite field $F_n$.  Thus every element $U \in F_n$ is a union of atoms which are obtained from the partition map  $Atoms_n$.  Thus $F_1$ has $Atoms_1: X \rightarrow \one$ given by the trivial partition of making no distinction among the elements of $X$, and hence $F_1 = \{X, \emptyset\}$.  The finite field $F_2$ is determined by a partition $Atoms_2: X \rightarrow \two$,  hence has two atoms, 
 say $U_1 \subset X$ and $U_2 \subset X$, and therefore the field $F_2 = \{X,U_1,U_2,\emptyset\}$.  The finite field $F_3$ is determined by a partition $Atoms_3: X \rightarrow \mathbf{3}$,  hence has three atoms, 
 say $U_1 \subset X$, $U_{2,1} \subset X$, and $U_{2,2} \subset X$, where $U_{2,1} \cup U_{2,2} = U_2$, and therefore the field $F_3$ has $2^3$ elements.    At each iteration, one of the partitions consisting of more than one element is refined by splitting it into two separate nonempty subsets.

To view the partitioning as a function to $\bN$ simply take the composite map
\begin{figure}[H]
 \begin{equation}  \nonumber
 \begin{tikzpicture}[baseline=(current bounding box.center)]
          \node     (X)  at   (-1,0)   {$X$};
          \node     (n) at   (2,0)   {$\mathbf{n}$};
          \node     (N)  at   (2, -1.5)  {$\bN$};

          \draw[->,above] (X) to node {$Atoms_n$} (n);
          \draw[>->,right] (n) to node {$inclusion$} (N);
          \draw[->,dashed,below] (X) to node [xshift=-0pt]{$\mathcal{A}_n$}  (N);
     \end{tikzpicture}   
 \end{equation}
 \caption{Every standard measurable space $X$ is generated by a countable family of maps $\A_n$ which partition $X$ into atoms.}
 \label{atomsDiagram}
 \end{figure}
\noindent
For brevity, denote the composite map as $\mathcal{A}_n$.  

Since $X$ is a standard space the finite fields $F_n$ which generate the field $\F$ on $X$ satisfy the property that every element $X_{n,i} := \mathcal{A}_n^{-1}(i)$ which is an atom of the field $F_n$ either remains an atom in $F_{n+1}$ or gets split into two separate atoms of the finite field $F_{n+1}$.   Thus $X_{n,i}$ is an atom of $F_{n+1}$ or $X_{n,i}$ gets split into two, say $X_{n,i_1}$ and $X_{n,i_2}$.  Using the isomorphism
\be \nonumber
\{0,1,2,\ldots, i-1,i_1,i_2, i+1,i+2,\ldots, n-1\} \cong \{0,1,2,\ldots, n\}
\ee
we have a monotonic decreasing function $\phi: \bN \rightarrow \bN$ specified by 
\be \nonumber
\begin{array}{ccccc}
\phi&:& \bN & \rightarrow & \bN \\
&:& k & \mapsto & \left\{ \begin{array}{ll} k & \textrm{for all }k \le i \\
i & \textrm{for }k=i+1 \\
k-1 & \textrm{for all }k>i+1
\end{array} \right.
\end{array}
\ee
such that the $\M$-diagram
\begin{figure}[H]
\begin{equation}   \nonumber
 \begin{tikzpicture}[baseline=(current bounding box.center)]
          \node     (X)  at   (-1.2,0)   {$(X,F_{n+1})$};
          \node     (X2) at   (1.2,0)   {$(X,F_{n})$};
          \node     (N1) at   (-1.2,-1.5)   {$(\bN, \powerset{\bN})$};
          \node     (N2)  at   (1.2, -1.5)  {$(\bN, \powerset{\bN})$};

         \draw[->,above] (X) to node {$\one_X$} (X2);
          \draw[->,below] (N1) to node {$\mS' \phi$} (N2);
          \draw[->,left] (X) to node [yshift=3pt]{$\mathcal{A}_{n+1}$} (N1);
          \draw[->,right] (X2) to node [yshift=3pt]{$\mathcal{A}_n$}  (N2);
     \end{tikzpicture}   
 \end{equation}
 \caption{The refinement process used in generating the finite fields associated with every standard measurable space satisfy the commutativity condition depicted in the diagram. }
 \label{refinement}
 \end{figure}
\noindent
commutes.

\item If $\{G_n\}_{n=1}^{\infty}$ is a sequence of atoms with $G_n \in F_n$ such that $G_{n+1} \subset G_n$ for $n=1,2,\ldots$, then $\bigcap_{n=1}^{\infty} G_n \ne \emptyset$. Since $G_n \in F_n$ is an atom it is given by $G_n = \A_n^{-1}(i)$ for some index $i \in \mathbf{n}$. 

\end{enumerate}

A basis for a field is an asymptotically generating sequence of finite fields with the property that a 
decreasing sequence of atoms cannot converge to the empty set.  Property (2) of a sequence 
of fields is called the finite intersection property. 

 
\begin{quote}
A measurable space $(X,\Sigma_X)$ is called standard if $\Sigma_X = \sigma(\F)$ for some field $\F$ which possesses a basis.
\end{quote}
\vspace{.1in}

Since the generating fields are  given to us it may be the case that some $F_n$ are omitted. The fact some $F_n$ may be omitted is immaterial; what matters is the \emph{refinement of the partition} at each step which makes the partitioning of $X$ finer as the indexing set increases.  Regardless of the $\mathbf{F}$ characterizing the finite field $\F$ which the sets $F_i$ generate, the number of atoms is a monotonically increasing sequence in $\bN$ as a function of the indexing set.  If the finite fields are given to us with say $F_n$ having $k$ atoms and $F_{n+1}$ having $k+m$ atoms then in Diagram \ref{refinement} the monotonic function $\phi$ will be the composite  of $m$ monotonic decreasing functions, which is a monotonic decreasing function.  Hence we will assume without loss of generality that the sequence of finite fields $\{F_n\}_{n=1}^{\infty}$ is such that $F_{n+1}$ has one more atom than the finite field $F_n$.

\section{The codense functor  $\mS':\Om \rightarrow \mathbf{Std}_2$}   \label{properties}

Let $\llceil \cN \rrceil$ denote the full subcategory of $\SCvx$ consisting of the single object $\cN$, and let $\iota_{\cN} : \llceil \cN \rrceil \hookrightarrow \Om$ denote the inclusion functor.

\begin{lemma} \label{natCond} 
Suppose that $\mathbf{U}: X \rightarrow \mS  \mathbbmss{n}$ is a measurable function with image $ \mathbbmss{n}$.  Then the composite map $X \stackrel{\mathbf{U}}{\longrightarrow}  (\mathbbmss{n}, \powerset{\mathbbmss{n}})  \hookrightarrow  (\bN,\powerset{\bN})$ is a measurable function,  and  
\be \nonumber
\textrm{ if }J \in Nat( \Std_2(X, \mS'(\iota_{\cN} \_)), \Std_2(\one, \mS'(\iota_{\cN} \_))) \textrm{ then }J(\mathbf{U}) \in  \mathbbmss{n}.
\ee

\end{lemma}
\begin{proof}  Proof by contradiction.  
Suppose, to obtain a contradiction, that $J_{\bN}(\mathbf{U}) = m$, with $m\ge n$ which does not lie in the image of $ \mathbbmss{n}$.  Choose a monotonic function $\phi:\cN \rightarrow \cN$ which sends all elements $k \in \cN$ such that $k \ge m$ to $m$, and all elements $k \in \cN$ such that $k<m$ get mapped to $0$.  Thus $\big(\phi \circ \mathbf{U}\big)(x) = \chi_X(x) \underline{0}  + \chi_{\emptyset}(x) \underline{m} = 1 \underline{0} + 0 \underline{m} = \underline{0}$, which is the constant function with value $0$.  The naturality condition $\phi( J_{\bN}(\mathbf{U})) = J_{\bN}( \phi \circ \mathbf{U})$, obtained by the monotonic mapping sending everything to zero,  requires $J_{\bN}(\overline{\underline{0}}) = \underline{0}$ where $\overline{\underline{0}}$ is the constant function.  The hypothesis $J_{\bN}(\mathbf{U})=m$ implies $\phi( J_{\bN}(\mathbf{U})) = \phi(m)=m$.  We thus have
\be \nonumber
 m = \phi(m) = \phi( J(\mathbf{U})) = J_{\bN}( \phi \circ \mathbf{U}) = J_{\bN}(\overline{\underline{0}})=0.
 \ee
 Hence we must conclude that $J(\mathbf{U}) \in  \mathbbmss{n}$.
\end{proof}

\begin{lemma} \label{NcoDense}  
 The composite functor $\llceil \cN \rrceil \stackrel{\iota_{\cN}}{\hookrightarrow} \Om \stackrel{\mS'}{\longrightarrow} \Std_2$ is a codense functor.
\end{lemma}
\begin{proof} 
The functor $\mS' \circ \iota_{\cN}$ is codense if and only if the  functor $\mathbf{y}: \Std_2^{op} \rightarrow \Set^{\llceil \cN \rrceil}$, specified on objects by $\mathbf{y}(X) = \Std_2(X,\mS'( \iota_{\cN}\_))$ is full and faithful.

The functor $\mathbf{y}$  is faithful follows because if $x_1, x_2 \in X$ then since $X$ is a separated measurable space there exists a measurable set $U$ in $X$ such that $x_1 \in U$ and $x_2 \not \in U$, and hence the characteristic function \mbox{$\chi_U: X \rightarrow \mathbf{2} \xrightarrow{sw_2} \two \hookrightarrow \bN$} suffices to separate the two points. More explicitly, the measurable function $\mathbf{U}:X \rightarrow \bN$ defined, for each $x \in X$ by $\mathbf{U}(x) =  \chi_{U}(x) \, \underline{0} + \chi_{U^c}(x) \underline{1}$ suffices.  
 
\vspace{.075in} 

 Since $\Std_2$ has  the object $\one$ as a separator, to prove that $\mathbf{y}$ is a full functor  it suffices to consider natural transformations  $J \in Nat( \Std_2(X, \mS'(\iota_{\cN} \_)), \Std_2(\one, \mS'(\iota_{\cN} \_)))$. 
 We proceed to show that  $J_{\bN} = ev_x$ for a unique point $x \in X$. 

For $n \ge 1$ let    $\mathcal{A}_n: X \rightarrow \bN$  be the  measurable function which partitions $X$, giving the atoms of the  field $F_n$ which are used to generate the $\sigma$-algebra on $X$. 
By Lemma \ref{natCond} it follows that 
 $J_{\bN}(\mathcal{A}_n) = k$ for some $k \in \mathbf{n}$, and this implies that $J_{\bN}(\mathcal{A}_{n}) =ev_{x}(\mathcal{A}_n)= \mathcal{A}_n(x)$ for every point $x \in \mathcal{A}_n^{-1}(k)$.  Let $W_n=\mathcal{A}_n^{-1}(k)$.

Next we compute $J_{\bN}(\mathcal{A}_{n+1})$, where $\mathcal{A}_{n+1}$ is a refinement of the partition $\mathcal{A}_n$.  By Lemma \ref{natCond} it follows that $J_{\bN}(\mathcal{A}_{n+1}) \in \mathbf{n}+1$.  By the relationship $\phi \circ \mathcal{A}_{n+1} = \mathcal{A}_n$ given in equation (\ref{refinement}), where $\phi$ is a monotonic decreasing function, it follows that  
\be \label{refinementCond}
\phi(J_{\bN}(\mathcal{A}_{n+1}) )= J_{\bN}(\phi \circ \mathcal{A}_{n+1}) = J_{\bN}(\mathcal{A}_n)=k.
\ee
If the atom $\mathcal{A}_n^{-1}(k)$ was partitioned into two separate atoms then $\phi^{-1}(k)=\{k, k+1\}$, and hence $J_{\bN}(\mathcal{A}_{n+1}) = k$ or $J_{\bN}(\mathcal{A}_{n+1}) = k+1$.  Let $k_{\star}$ denote either the index $k$ or $k+1$, depending upon which of the equations is true.
 This implies that $J_{\mathbb{N}}(\mathcal{A}_{n+1}) =ev_{x}(\mathcal{A}_{n+1})= \mathcal{A}_{n+1}(x)$ for every point $x \in \mathcal{A}_{n+1}^{-1}(k_{\star})$.  Let $W_{n+1}=\mathcal{A}_{n+1}^{-1}(k_{\star}) \subseteq \mathcal{A}_n^{-1}(k) = W_n$.
 
Now suppose the atom $\mathcal{A}_n^{-1}(k)$ was not partitioned and $J_{\mathbb{N}}(\mathcal{A}_{n+1}) = m$, where by Lemma \ref{natCond} it follows that $m \in \mathbf{n}+1$.  Equation \ref{refinementCond} always holds, regardless of where the given partition refinement of an atom occurs, and hence it follows that
\be \nonumber
m=J_{\bN}(\mathcal{A}_{n+1}) = \left\{ \begin{array}{ll}
k & \textrm{if the partition refinement occurs at an atom with index }>k \\
k+1 & \textrm{if the partition refinement occurs at an atom with index }<k
\end{array}  \right..
\ee
Let $W_{n+1}=\mathcal{A}_{n+1}^{-1}(k_{\star})$, where $k_{\star}$ is either $k$ or $k+1$ depending upon where the partition refinement occured.  Regardless of where the partition refinement occurred we always have the property that $W_{n+1} \subset W_{n}$.

 We can continue this process by partitioning the map $\mathcal{A}_{n+2}$  just as we partitioned the map $\mathcal{A}_{n+1}$.  In this manner we obtain a monotone decreasing sequence of sets $W_n \supset W_{n+1} \supset W_{n+2} \supset \ldots$ such that $J_{\bN}(\mathcal{A}_n) = ev_x(\mathcal{A}_n)$ for every $x \in W_n$.  Since $X$ is a standard measurable space and the atoms satisfy the finite intersection property  it follows that $Z = \bigcap_{i=1}^{\infty} W_i \ne \emptyset$. Since $X$ is a separated measurable space  the  set $Z$ must be a singleton set $\{x\}$, and it follows that $J_{\mathbb{N}} = ev_x$.   
\end{proof}

We say that a measurable function $\mathbf{U}: X \rightarrow \G{\mathbb{N}}$ is \textbf{deterministic} if and only if for every $x \in X$ it follows that  $\mathbf{U}(x) := \sum_{i \in \bN} \mathbf{U}_i(x) \, \delta_i$, with  $\sum_{i \in \bN} \mathbf{U}_i(x)=1$, satisfies  $\mathbf{U}(x) = \delta_{\kappa(x)}$ for some index $\kappa(x) \in \Nat$.  The function $\kappa: X \rightarrow \Nat$ generally varies with the point $x \in X$.

\begin{lemma} \label{JNdeterministic}  Let  $X$ be an object in $\Std_2$, and let \mbox{$\mathbf{U}: X \rightarrow \G{\bN}$} be a measurable  function \mbox{$\mathbf{U}(x) = \sum_{i\in \Nat} \mathbf{U}_i(x) \delta_i$} where, for every $x \in X$, \mbox{$\sum_{i\in \Nat} \mathbf{U}_i(x) = 1$}.
Then for every natural transformation  \mbox{$J \in Nat( \Std_2(X, \mS'), \Std_2(\one, \mS'))$} and every permutation $\phi: \Nat \rightarrow \Nat$ it follows that  if $J_{\G{\bN}}(\mathbf{U}) = \sum_{i\in \Nat} p_i \, \delta_i$ then
$J_{\bN}(\epsilon \circ \G(\phi) \circ \mathbf{U}) = min_{i} \{ \phi(i) \, | \, p_i>0 \}$. 

\end{lemma}
\begin{proof}
By naturality  $J_{\G{\bN}}(\G(\phi) \circ \mathbf{U}) = \G(\phi) J_{\G{\bN}}(\mathbf{U})$.  Since $J_{\G{\bN}}(\mathbf{U}) \in \G{\bN}$ it can be written as a countably affine sum of the Dirac measures,  $J_{\G{\bN}}(\mathbf{U})=\sum_{i\in \Nat} p_i \delta_i$.  Consequently we have
\be \nonumber
\begin{array}{lcl}
J_{\bN}( \epsilon_{\bN} \circ \G(\phi) \circ \mathbf{U}) &=& \epsilon_{\bN}( \G(\phi) (J_{\G{\bN}}(\mathbf{U}))) \\
 &=& \epsilon_{\bN}( \displaystyle{ \sum_{i\in \Nat} }p_i \delta_{\phi(i)} )\\
&=& \displaystyle{ \sum_{i\in \Nat} }p_i \underline{\phi(i)} \quad \textrm{as a countable affine sum in }\cN\\
&=& min_{i} \{ \phi(i) \, | \, p_i>0 \} 
\end{array}.
\ee 
where the last equality follows from the super convex space structure on $\cN$.
\end{proof}

\begin{lemma} \label{mScodense}  The functor $\mS': \Om \rightarrow \Std_2$ is a codense functor.
\end{lemma}
\begin{proof}
Let $\mathbf{y}: \Std_2^{op} \rightarrow \Set^{\Om}$ be given by $X \mapsto \Std_2(X,\mS')$.   We have the $\mathbf{Cat}$-diagram 
\begin{equation}   \label{deterministicSq}
 \begin{tikzpicture}[baseline=(current bounding box.center)]
          \node   (S) at  (-.3,0)  {$\Std_2^{op}$};
          \node   (Om)   at   (3,1.5)    {$\Set^{\Om}$};        
          \node    (N)  at    (3, 0)  {$\Set^{\llceil\cN \rrceil}$};
          
          \draw[->,above] (S) to node {$\mathbf{y}$} (Om);
          \draw[->,below] (S) to node {\small{$\Set^{\iota_{\cN}} \circ \mathbf{y}$}} (N);
          \draw[->,right] (Om) to node {$\Set^{\iota_{\cN}}$} (N);
          
     \end{tikzpicture},    
 \end{equation}
where we know, from Lemma \ref{NcoDense}, that for  every measurable space $X$ that the composite $\Set^{\iota_{\cN}} \circ \mathbf{y}$ satisfies the property that  every natural transformation 
\be \nonumber
J \in Nat( \Std_2(X, \mS'(\iota_{\cN}\_)), \Std_2(\one,\mS'(\iota_{\cN}\_)))
\ee
 is, at component $\bN$, given by  $J_{\bN} = ev_x$ for some $x \in X$.
 
Let  $X \stackrel{\mathbf{U}}{\longrightarrow} \G{\bN}$  be any measurable function, hence for each $x \in X$ we have the countably affine sum $\mathbf{U}(x) = \sum_{i\in \Nat} \mathbf{U}_i(x) \delta_i$. 
 By naturality the $\Set$ diagram
    \begin{equation}   \label{deterministicSq}
 \begin{tikzpicture}[baseline=(current bounding box.center)]
          \node   (GNX) at  (0,0)  {$\G{(\mathbb{N})}^X$};
          \node   (GN)   at   (3,0)    {$\G{\mathbb{N}}$};        
          \node    (XN)  at    (0, -1.5)  {$\mathbb{N}^X$};
          \node    (N)    at   (3, -1.5)    {$\mathbb{N}$};
          
          \draw[->,above] (GNX) to node {$J_{\G(\mathbb{N})}$} (GN);
          \draw[->>,right] (GN) to node {$\epsilon_{\mathbb{N}}$} (N);
          \draw[->>,left] (GNX) to node [xshift=0pt,yshift=0pt] {$\epsilon_{\bN}^X$} (XN);
          \draw[->,below] (XN) to node {$J_{\mathbb{N}}$} (N);
          
           \node   (U) at  (4.5,0)  {$\mathbf{U}$};
          \node   (JU)   at   (8.9,0)    {$J_{\G{\mathbb{N}}}(\mathbf{U})$};        
          \node    (k)  at    (4.5, -1.5)  {$\epsilon_{\bN} \circ \mathbf{U}$};
          \node    (Jk)    at   (8, -1.5)    {\small{$J_{\bN}(\epsilon_{\bN} \circ \mathbf{U})=\epsilon_{\mathbb{N}}(J_{\G{\mathbb{N}}}(\mathbf{U}))$}};
          \node   (ph)   at    (8.9,-1.3)   {};
          
          \draw[|->] (U) to node {} (JU);
          \draw[|->] (JU) to node {} (ph);
          \draw[|->] (U) to node {} (k);
          \draw[|->] (k) to node {} (Jk);

     \end{tikzpicture},    
 \end{equation}
commutes.   Suppose that $J_{\bN}(\epsilon_{\bN} \circ \mathbf{U}) = ev_x(\epsilon_{\bN} \circ \mathbf{U}) = n$, and that $J_{\G{\mathbb{N}}}(\mathbf{U}) = \sum_{i\in \Nat} p_i \delta_i$.  Since $\epsilon_{\bN}(\sum_{i\in \Nat} p_i \delta_i) = \sum_{i\in \Nat} p_i \underline{i}$ the equation in the bottom right hand corner of the above diagram forces the coefficients $p_i=0$ for $i=0,1,\ldots,n-1$.    We claim that $J_{\G{\bN}}(\mathbf{U}) = \delta_n  \in  \epsilon_{\mathbb{N}}^{-1}(n) $. 

To obtain a contradiction suppose that $J_{\G(\bN)}(\mathbf{U})$ is not deterministic, and hence there exists an $m>n$ such that $p_m>0$.  Let $\phi:\Nat \rightarrow \Nat$ be the simple permutation interchanging the two elements $n$ and $m$, which yields the countably affine map $\G(\phi) : \G{\bN} \rightarrow \G{\bN}$.  By Lemma \ref{JNdeterministic} it follows that $J_{\bN}(\epsilon_{\bN} \circ \G(\phi) \circ \mathbf{U}) = m$, whereas, under the hypothesis that $J_{\G{\bN}}(\mathbf{U}) = \sum_{i=n}^{\infty} p_i \delta_i$ is nondeterministic,
\be \nonumber
\epsilon_{\bN}(J_{\G{\bN}}(\G(\phi) \circ \mathbf{U})) =\epsilon_{\bN}\bigg( \G(\phi) \big(J_{\G{\bN}}(\mathbf{U})\big) \bigg) = \epsilon_{\bN}( \sum_{i\in \Nat} p_i  \delta_{\phi(i)}) = \sum_{i=n}^{\infty} p_i  \underline{\phi(i)} = n.
\ee
Thus, to avoid a contradiction, we must conclude that $J_{\G{\bN}}(\mathbf{U}) = \delta_n = \mathbf{U}(x) = ev_x(\mathbf{U})$.

\end{proof}

The preceding lemma can be recast as 

\begin{corollary} \label{YStd}
The restricted Yoneda functor $\mathcal{Y}: \Std_2^{op} \rightarrow \Set^{\Om}$ is a full and faithful functor.  
\end{corollary}
\begin{proof}
See MacLane\cite[Proposition 2, p242]{Mac}. 
\end{proof}

\section{Constructing barycenter maps}   \label{sec:barycenter} 

\begin{lemma}  \label{epsilonExists} For every super convex space $A \in_{ob} \SCvx_{\star}$ there exists a measurable map, called the barycenter map, 
\be \nonumber
 \epsilon_{A} : \G(\mS A) \rightarrow \mS A
 \ee
such that 
\begin{enumerate}
\item $\epsilon_A \circ \eta_{\mS A} = \one_{\mS A}$, and
\item $\epsilon_A$ satisfies $\epsilon_A \circ \G(\epsilon_A) = \epsilon_A \circ \epsilon_{\G(\mS A)}$ where $\epsilon_{\G(\mS A)}:= \mu_{\mS A}$, and
\item $\epsilon_A$ is a countably affine map.
\end{enumerate}
Taken together, the first two properties are equivalent to saying $\epsilon_A$ is a $\G$-algebra.
\end{lemma}
\begin{proof}
Given $A$ apply the functor $\mS$, and the separation monad if necessary, to obtain the separated standard measurable space $\mS A$, and using Lemma \ref{NcoDense} represent $\mS A$  as the canonical colimit of the functor 
\be \nonumber
\D=  (\mS A \! \downarrow \! \iota_{\bN}) \stackrel{\pi}{\longrightarrow} \llceil \bN \rrceil \stackrel{\iota_{\bN}}{\longrightarrow} \Om \stackrel{\mS'}{\longrightarrow} \Std_2, 
\ee
so we have $\lim \D = (\mS A, \Std(\mS A, \bN))$ 
 where the universal projection arrows are given by the set of all measurable functions $f: \mS A \rightarrow \bN$.   We can construct a cone over $\D$ with vertex $\G(\mS A)$ and natural transformation $\omega$ from the constant functor assigning the object $\G(\mS A)$ to every component of $\D$  by specifying for each measurable function $f: \mS A \rightarrow  \bN$ that $\omega_f: \G(\mS A) \rightarrow  \bN$ be the measurable function $\epsilon_{\bN} \circ \G{f}$.  The $\M$-diagram\footnote{All the objects in the diagram are separated measurable spaces so whether we say ``$\Std_2$-diagram or $\Std$-diagram is a matter of choice.} is
  \begin{equation}   \nonumber
 \begin{tikzpicture}[baseline=(current bounding box.center)]
      \node   (PA)  at  (-.8,0)   {$\G(\mS A)$};
      \node   (A)     at  (2.4,0)   {$\mS A$};
      \node   (PN1) at  (2,1.5)  {$\G( \bN)$};
      \node   (PN2) at  (2, -1.5)  {$\G(\bN)$};

      \node   (N1)   at  (5, 1.5)    {$\bN$};
      \node   (N2)   at  (5, -1.5)   {$\bN$};
      \draw[->,above] (PA) to node {\small{$\G{f}$}} (PN1);
      \draw[->,above] (PN1) to node {\small{$\epsilon_{\bN}$}} (N1);
      \draw[->,right] (N1) to node {\small{$\mS \phi$}} (N2);
      \draw[->,below] (PA) to node {\small{$\G{g}$}} (PN2);
      \draw[->,below] (PN2) to node {\small{$\epsilon_{\bN}$}} (N2);
      \draw[->,left, dashed] (PN1) to node [yshift=17pt]{\tiny{$\G(\mS \phi)$}} (PN2);
      
      \draw[->,above] (A) to node {$f$} (N1);
      \draw[->,below] (A) to node {$g$} (N2);
      
   \end{tikzpicture}
 \end{equation}
 \noindent
 $(\G(\mS A), \omega)$ specifies a cone over $\D$ because if $g = \mS \phi \circ f$, where $\phi: \cN \rightarrow \cN$, then $\G{g} = \G(\mS \phi \circ f)$, and hence
 \be \nonumber
 \begin{array}{ccll}
 \mS \phi \circ \omega_f &=& \mS \phi \circ \epsilon_{\bN} \circ \G{f} & \textrm{by def. of }\omega_f \\
 &=&  (\epsilon_{\bN} \circ \G(\mS \phi)) \circ \G{f} & \textrm{by naturality }\epsilon_{\bN} \circ \G(\mS \phi) = \mS \phi \circ \epsilon_{\bN}  \\
 &=&  \epsilon_{\bN}  \circ \G( \mS \phi \circ f) & \textrm{factoring out common }\G\\
  &=&  \epsilon_{\bN}  \circ \G{g}  & \textrm{by hypothesis }g= \mS \phi \circ f\\
 &=& \omega_g   & \textrm{by def. of }\omega_g
 \end{array}
 \ee
 Because $\mS A$ is the limit of the diagram $\D$ it follows by universality that there exists a unique measurable map $\epsilon_{A}: \G(\mS A) \rightarrow \mS A$ such that $\omega_f = \epsilon_{A} \circ f$ for all measurable maps $f: \mS A \rightarrow \mS \bN$.

\vspace{.1in}
 \noindent
 Property (1):
 To show that $\epsilon_{A}(\delta_a) = a$ for all $a \in A$ note that at  each component $f$ we have 
 \be \nonumber
 \begin{array}{lcll}
\big( \epsilon_{\bN} \circ \G{f} \big)(\delta_a)  &=& \epsilon_{\bN}(\delta_{f(a)}) & \textrm{because }\G{f}(\delta_a)=\delta_a f^{-1} = \delta_{f(a)}\\
 &=& f(a) & \textrm{property of }\epsilon_{\bN} 
 \end{array}
 \ee
 and we also have
 \be \nonumber
 \begin{array}{lcll}
\big( \epsilon_{\bN} \circ \G{f} \big)(\delta_a)  &=&  f( \epsilon_A(\delta_a)) & \textrm{because }\epsilon_{\bN} \circ \G{f} = f \circ \epsilon_A 
 \end{array}.
 \ee
 The space $\mS A$ is a separated measurable space so given any two distinct points $a,b \in A$  there exists a measurable set $U$ in $\mS A$  such that $a \in U$ and $b \not \in  U$, and hence the function $\chi_U$ coseparates the points.  Since the equaton $\epsilon_{\bN} \circ \G{f} = f \circ \epsilon_{A}$ must hold at every $f: \mS A \rightarrow \mS \bN$, including $f=\chi_U$, it follows that $\epsilon_A(\delta_a) = a$.
The fact that $\eta_{\mS A}(a)=\delta_a$ therefore yields  $\epsilon_A \circ \eta_{\mS A} = \one_{\mS A}$.

\vspace{.1in}
\noindent
Property (2):  
Note that the $\M$-diagram 
 \begin{equation}   \nonumber
 \begin{tikzpicture}[baseline=(current bounding box.center)]
      \node (GA2) at  (3, 1.5)  {$\G^2(\mS A)$};
      \node   (GN2) at  (6, 1.5) {$\G^2(\bN)$};
      \node   (GA) at  (3,0)  {$\G(\mS A)$};
      \node   (A) at  (3, -1.5)  {$\mS A$};
      \node   (gN) at  (6, 0)   {$\G(\bN)$};
      \node   (N)   at   (6, -1.5)  {$\bN$};

      \draw[->,above]    (GA2) to node {$\G^2(f)$} (GN2);
      \draw[->,left] ([xshift=-2pt] GA2.south) to node {\small{$\G(\epsilon_A)$}} ([xshift=-2pt] GA.north);
      \draw[->,right] ([xshift=2pt]  GA2.south) to node {\small{$\epsilon_{\G(\mS A)}$}} ([xshift=2pt] GA.north); 
      
      \draw[->,left] ([xshift=-2pt] GN2.south) to node {\small{$\G(\epsilon_\bN)$}} ([xshift=-2pt] gN.north);
      \draw[->,right] ([xshift=2pt]  GN2.south) to node {\small{$\epsilon_{\G(\bN)}$}} ([xshift=2pt] gN.north); 
  
      \draw[->,below] (GA) to node {$\G(f)$} (gN);
      \draw[->,right] (gN) to node {$\epsilon_{\bN}$} (N);
      \draw[->,below] (A) to node {$f$} (N);
      \draw[->,right] (GA) to node {$\epsilon_A$}(A);

   \end{tikzpicture}
 \end{equation}
has both squares in the top diagram commutative because 
\begin{enumerate}
\item naturality of $\mu$ (where $\epsilon_{\G(\mS A)} = \mu_{\mS A}$ and $\epsilon_{\G(\mS \bN)} = \mu_{\mS \bN}$ - these are the free $\G$-algebras), and \item the other square is just  the functor $\G$ applied to the bottom square which 
commutes for every measurable function $f: \mS A \rightarrow \mS \bN$. 
\end{enumerate}
  Hence, using the fact that $\epsilon_{\bN}$ is a $\G$-algebra, it follows that for all measurable $f: \mS A \rightarrow \bN$  that $f \circ \epsilon_A \circ \G(\epsilon_A) = f \circ \epsilon_A \circ \epsilon_{\G(\mS A)}$.
Because $\mS A$ is the canonical limit of the diagram $\D$ it then follows by the uniqueness property that $\epsilon_A \circ \G(\epsilon_A) = \epsilon_A \circ \epsilon_{\G(\mS A)}$.

\vspace{.1in}
\noindent
Property (3):
The property of $\epsilon_A$ being countably affine is equivalent to the statement that the
 $\M$-diagram
 \begin{equation}   \nonumber
 \begin{tikzpicture}[baseline=(current bounding box.center)]
          \node   (GN)  at  (-1,0)  {$\G(\bN)$};
          \node    (GA)  at  (2,1.5)  {$\G(\mS A)$};
          \node    (A)    at   (2,0)    {$\mS A$};
          \draw[->,above] (GN) to node [xshift=-17pt]{\small{$\langle \mathbf{P} \rangle = \langle P_i \rangle$}} (GA);
          \draw[->,right] (GA) to node {$\epsilon_A$} (A);
          \draw[->,below] (GN) to node {$\langle \epsilon_A(P_i) \rangle$} (A);
   \end{tikzpicture}
 \end{equation}
commutes because 
\be \nonumber
\big(\epsilon_A \circ \langle \mathbf{P} \rangle\big)( \sum_{i \in \Nat} q_i \delta_i) = \epsilon_A\big( \sum_{i \in \Nat} q_i P_i \big) \quad \textrm{ whereas } \quad\langle \epsilon_A(P_i) \rangle\big( \sum_{j \in \Nat} q_j \delta_j) = \sum_{i \in \Nat} q_j \, \epsilon_A(P_i). 
\ee

Now for every measurable $f: \mS A \rightarrow  \bN$ the square on the right hand square in the $\M$-diagram
 \begin{equation}   \nonumber
 \begin{tikzpicture}[baseline=(current bounding box.center)]

      \node   (GA) at  (3,0)  {$\G(\mS A)$};
      \node   (A) at  (3, -1.5)  {$\mS A$};
      \node   (gN) at  (6, 0)   {$\G{\bN}$};
      \node   (N)   at   (6, -1.5)  {$\bN$};
      
      \node  (GN) at  (0,0)   {$\G{\bN}$};

      \draw[->,above]    (GN) to node {$\langle \mathbf{P} \rangle$} (GA);
      \draw[->,below,left] (GN) to node [yshift=-3pt]{$\langle \epsilon_A(P_i) \rangle$} (A);
      \draw[->,out=45,in=135,looseness=.7,dashed,above] (GN) to node {\small{$\langle P_if^{-1} \rangle = \G(f) \circ \langle \mathbf{P} \rangle$}} (gN);
  
      \draw[->,above] (GA) to node {$\G(f)$} (gN);
      \draw[->,right] (gN) to node {$\epsilon_{\bN}$} (N);
      \draw[->,below] (A) to node {$f$} (N);
      \draw[->,right] (GA) to node {$\epsilon_A$}(A);
   \end{tikzpicture}
 \end{equation}
\noindent
commutes because $\epsilon_A$ is the unique arrow from $\G(\mS A)$ to $\mS A$ satisfying $f \circ \epsilon_A = \epsilon_{\bN} \circ \G(f)$.  Thus we obtain a cone over $\D$ 
with vertex $\G{\bN}$ and components $\epsilon_{\bN} \circ \langle P_i f^{-1}\rangle$ which is a countably affine map for all $f$. Since the countably affine map $\epsilon_{\bN} \circ \langle P_if^{-1} \rangle$ is uniquely specified by where they map the elements $\delta_i$, we note that
\be \nonumber 
\begin{array}{lcl}
\big(\epsilon_{\bN} \circ \langle P_if^{-1} \rangle\big)(\delta_j) &=&  \epsilon_{\bN} \circ P_jf^{-1} \\
&=& (\epsilon_{\bN} \circ \G(f))(P_j) \\
&=& (f \circ \epsilon_A)(P_j) \\
&=& f(\epsilon_A(P_j)) \\
&=& (f \circ \langle \epsilon_A(P_i) \rangle)(\delta_j)
 \end{array} 
\ee
But we can also write
\be \nonumber
\begin{array}{lcl}
\big(\epsilon_{\bN} \circ \langle P_if^{-1} \rangle\big)(\delta_j) =  (\epsilon_{\bN} \circ \G(f) \circ \langle \mathbf{P}\rangle) (\delta_j)
&=& (f \circ \epsilon_A \circ \langle \mathbf{P} \rangle)(\delta_j)
\end{array}.
\ee

Since these two equations hold for all $j \in \bN$ it follows that 
\be \nonumber
f \circ \langle \epsilon_A(P_i) \rangle = \epsilon_{\bN} \circ \langle P_i f^{-1} \rangle = f \circ \epsilon_A \circ \langle \mathbf{P} \rangle,
\ee
and it therefore follows by the uniqueness property of $\lim \D$  that  $\langle \epsilon_A(P_i) \rangle = \epsilon_A \circ \langle \mathbf{P} \rangle$.

\end{proof}

\begin{example} The coequalizer of the two points $\frac{1}{3}:\one \rightarrow [0,1]$ and $\frac{2}{3}:\one \rightarrow [0,1]$ is the three point space $A=\{\underline{0}, \underline{u}, \underline{1}\}$ with the structure defined by $r \underline{0} + (1-r) \underline{u} = \underline{u}$ for all $r \in (0,1)$, $r \underline{1} + (1-r) \underline{u} = \underline{u}$ for all $r \in (0,1)$, and $r \underline{0} + (1-r) \underline{1} = \underline{u}$ for all $r \in (0,1)$.  There are two basic map  $m_1, m_2: A \rightarrow \cN$ in $\Ai$ defined by 
\begin{enumerate}
\item $m_1(\underline{0})=m_1(\underline{u})=0$ and $m_1(\underline{1})=1$.
\item $m_2(\underline{1})=m_2(\underline{u})=0$ and $m_2(\underline{0})=1$.
\end{enumerate}
All other maps $A \rightarrow \cN$ are obtained from one of these two maps by composition with a countably affine map $\phi:\cN \rightarrow \cN$.
Hence $Ran_{\iota}(\mS')(A) \cong (\twos,\powerset{\twos})$.  The barycenter map is given by 
\be \nonumber
\begin{array}{lcl}
\T(\twos) & \stackrel{\epsilon_A}{\longrightarrow} & A \\
(1-r) \delta_0 + r \delta_1 & \mapsto & \left\{ \begin{array}{ll} \underline{0} & \textrm{iff } r=0 \\ \underline{1} & \textrm{iff }r=1 \\ \underline{u} & \textrm{iff }r \in (0,1) \end{array} \right.
\end{array}
\ee 
\end{example}

\section{The adjunction $\T \dashv \mS$ and isomorphism $\Std^{\G} \cong \SCvx_{\star}$}

 \begin{lemma} \label{epsilonNT}  The family of maps $\epsilon_A: \T(\mS A) \rightarrow A$, one for each countably generated  super convex space $A$  form the components of a natural transformation  $\epsilon: \T \circ \mS \Rightarrow \one_{\SCvx_{\star}}$.
\end{lemma} 
 \begin{proof}  Suppose $m: A \rightarrow B$ is a countably affine map.  Using Lemma \ref{epsilonExists} construct the countably affine measurable maps $\epsilon_A$ and $\epsilon_B$.  Let $\D$ be the diagram 
 \be \nonumber
\D=  (\mS A \! \downarrow \! \iota_{\bN}) \stackrel{\pi}{\longrightarrow} \llceil \bN \rrceil \stackrel{\iota_{\bN}}{\longrightarrow} \Om \stackrel{\mS'}{\longrightarrow} \Std_2, 
\ee
 and let $\E$ be the diagram 
 \be \nonumber
\E =  (\mS B \! \downarrow \! \iota_{\bN}) \stackrel{\pi}{\longrightarrow} \llceil \bN \rrceil \stackrel{\iota_{\bN}}{\longrightarrow} \Om \stackrel{\mS'}{\longrightarrow} \Std_2. 
\ee

 For every measurable function $f: \mS B \rightarrow \mS \bN$ the composite function $f \circ \mS m: \mS A \rightarrow \mS \bN$ is a projection 
 arrow in the limit of $\D$, and hence because \mbox{$(\G(\mS A), \{\epsilon_{\bN} \circ \G(f') \, | \, \textrm{for all }f': \mS A \rightarrow \mS \bN\})$}  is a cone over $\D$ the unique arrow \mbox{$\epsilon_A: \G(\mS A) \rightarrow \mS A$} satisfies \mbox{$(f \circ \mS m) \epsilon_A = \epsilon_{\bN} \circ \mS \T(f \circ \mS m)$}.  Hence the outer square of the $\Std_2$ diagram 
 
 \begin{equation}   \nonumber
 \begin{tikzpicture}[baseline=(current bounding box.center)]
        \node  (GA)     at   (-1,1.5)  {$\G(\mS A)$};
       \node  (A)   at  (-1,0)   {$\mS A$};
       \node (GB)    at   (2,1.5)    {$\G(\mS B)$};
       \node (B)      at    (2, 0)   {$\mS B$};
       \node   (N)  at   (5,0)     {$\mS \bN$};
       \node   (GN) at  (5, 1.5)   {$\G(\mS \bN)$};
       
       \draw[->,left] (GA) to node {\small{$\epsilon_A$}} (A);
       \draw[->,above] (GA) to node {\small{$\G(\mS m)$}} (GB);
       \draw[->,left] (GB) to node {$\epsilon_B$} (B);
       \draw[->,below] (A) to node {$\mS m$} (B);       
       \draw[->,above] (GB) to node {$\G(f)$} (GN);
       \draw[->,right] (GN) to node {$\epsilon_{\bN}$} (N);
       \draw[->,below] (B) to node {$f$} (N);
  \end{tikzpicture}
\end{equation}
\noindent
commutes, and the right square commutes because $\epsilon_B: \G(\mS B) \rightarrow \mS B$ is the unique arrow from the vertex of  the cone $(\G(\mS B), \{\epsilon_{\bN} \circ \G{f} \, | \, f: \mS B \rightarrow \mS \bN\}$ over $\E$ to the $\lim \E = (\mS B, \{f: \mS B \rightarrow \mS \bN\})$.
 
 Thus we have the two equations, $f \circ \mS m \circ \epsilon_A = \epsilon_{\bN} \circ \G(f) \circ \G(\mS m)$ and $f \circ \epsilon_B = \epsilon_{\bN} \circ \G(f)$.  Using the second equation and substituting into the first equation, replacing the expression $\epsilon_{\bN} \circ \G(f)$, we obtain $\epsilon_B \circ \G(\mS m) = \mS m \circ \epsilon_A$ which proves naturality.
\end{proof}

\begin{thm} \label{adj} The pair of functors $\T: \Std \rightarrow \SCvx_{\star}$ and $\mS: \SCvx_{\star} \rightarrow \Std$ specify an adjunction $\langle \T, \mS, \eta, \epsilon \rangle$ with $\T \dashv \mS$.
\end{thm}
\begin{proof} 
The unit of the adjunction is $\eta_X(x) = \delta_x$ while the counit of the adjunction is the natural transformation $\epsilon$ specified in Lemma \ref{epsilonNT}. 

We verify  the two triangular identities.  For $X$ any measurable space we have the commutative $\SCvx$-diagram,
\begin{equation}   \nonumber
 \begin{tikzpicture}[baseline=(current bounding box.center)]
 
 \node        (PX)  at  (0, 0)   {$\T(X)$};
  \node     (PPX) at  (3,0)  {$\T(\Sigma \T(X))$};
   \node   (TX)  at   (3,-1.1)   {$\T(X)$};
   
     \draw[->,above] (PX) to node {\small{$\T\eta_X$}} (PPX);
  \draw[->,below,left] (PX) to node [yshift=-2pt] {\small{$id_{\T X}$}} (TX);
  \draw[->,right] (PPX) to node {\small{$\epsilon_{\T(X)}$}} (TX);
  
  \node        (p)  at  (5, 0)   {$P$};
  \node     (dp) at  (8,0)  {$\delta_P$};
   \node   (p2)  at   (8,-1.1)   {$P$};
   
     \draw[|->] (p) to node {} (dp);
  \draw[|->] (dp) to node  {} (p2);
  \draw[|->] (p) to node {} (p2);
  
 \end{tikzpicture}
 \end{equation}
\noindent
and for $A$ any countably generated super convex space, we have the commutative $\Std$-diagram,
\begin{equation}   \nonumber
 \begin{tikzpicture}[baseline=(current bounding box.center)]
 
 \node        (SPSA)  at  (0, 1.)   {$\Sigma(\T\Sigma A)$};
  \node     (SA) at  (0,0)  {$\Sigma A$};
   \node   (SA2)  at   (3,1.)   {$\Sigma A$};
   
     \draw[->,left] (SA) to node {\small{$\eta_{\Sigma A}$}} (SPSA);
  \draw[->,below,right] (SA) to node  [yshift=-3pt]{\small{$id_{\Sigma A}$}} (SA2);
  \draw[->,above] (SPSA) to node {\small{$\mS \epsilon_A$}} (SA2);
  
  \node        (a)  at  (5, 0)   {$a$};
  \node     (da) at  (5,1.)  {$\delta_a$};
   \node   (a2)  at   (8,1.)   {$a$};
   
     \draw[|->] (a) to node {} (da);
  \draw[|->] (da) to node  {} (a2);
  \draw[|->] (a) to node {} (a2);
  
 \end{tikzpicture}
 \end{equation}
where in both triangular identities we have used the property that, for all spaces $A$ and all elements $a \in A$,  $\epsilon_A(\delta_a)=a$ as shown in Lemma \ref{epsilonExists}.

Thus $\langle \T, \mS, \eta, \epsilon \rangle$ specifies an adjunction with $\T \dashv \mS$.
\end{proof}

\begin{corollary} The adjunction $\langle \T, \mS, \eta, \epsilon \rangle$ is an adjoint factorization of $\G$.
\end{corollary}
\begin{proof} This follows from Theorem \ref{adj} and Lemma \ref{sigmaAlgebra}.
\end{proof}

Because each countably affine map $m: A \rightarrow \Rinf$ yields a measurable function, $\mS m$, we have

 \begin{corollary}  Let $A \in \SCvx_{\star}$.   Viewing a probability measure as a functional, 
 \be \nonumber
 \begin{array}{ccccc}
 \hat{P} &:& \Std_2(\Sigma A, \Rinf) & \longrightarrow & \Rinf \\
 &:& f & \mapsto & \int_A f \, dP
 \end{array}
 \ee
  we have the result that the restriction of $P \in \G(\mS A)$ to the countably affine (measurable) functions $\SCvx(A, \Rinf)$ is an evaluation map, $P| = ev_a$, for a unique point $a \in A$.  In other words,  every $m \in \SCvx(A, \Rinf)$ is a measurable function $\mS m \in \Std_2(\mS A, \Rinf)$, and for every such $m$ it follows that $\hat{P}(m)=m(a)$ for a unique point $a \in A$. 
\end{corollary}
\begin{proof}
This result is a  translation of the naturality of $\epsilon$.  For every $m \in \SCvx(A, \Rinf)$ the square 
 \begin{equation}   \nonumber
 \begin{tikzpicture}[baseline=(current bounding box.center)]

          \node   (PA)   at   (-2,1.7)    {$\T(\Sigma A)$};
          \node    (PR) at     (1.5,1.7)  {$\T(\Sigma \Rinf)$};
          \node    (A)    at     (-2, 0)    {$A$};
           \node   (R)  at   (1.5,0)  {$\Rinf$};
           
           \draw[->,above] (PA) to node {$\T(\Sigma m)$} (PR);
           \draw[->,below] (A) to node {$m$} (R);
           \draw[->,left] (PA) to node {$\epsilon_A$} (A);
           \draw[->,right] (PR) to node {$\epsilon_{\Rinf}=\mathbb{E}$} (R);
           
           \node  (P)  at   (4, 1.7)   {$P$};
           \node   (Pm) at  (8.9, 1.7)  {$Pm^{-1}$};
           \node   (eP) at  (4,0)  {$\epsilon_A(P)$};
           \node    (E)  at  (8.9,0)   {$\epsilon_{\Rinf}(Pm^{-1})$};
           \node    (ma) at  (6.7,0)  {$m(\epsilon_A(P))=$};
           
           \draw[|->] (P) to node {} (Pm);
           \draw[|->] (P) to node {} (eP);
           \draw[|->] (Pm) to node {} (E);
           \draw[->] (eP) to node {} (ma);
    
         \end{tikzpicture}
 \end{equation}
 commutes.  Using the fact that $\epsilon_{\Rinf}$ is the expectation operator, $\mathbb{E}$,  we have
 \be \nonumber
ev_{\epsilon_A(P)}(m) = m(\epsilon_A(P)) =  \epsilon_{\Rinf}(Pm^{-1})= \mathbb{E}(Pm^{-1}) = \int_{\Rinf} id_{\Rinf} d(Pm^{-1}) = \int_A m \, dP.
\ee
For a given $P \in \G(\mS A)$, the uniqueness property is simply the statement that $\epsilon_A(P) = a$.
Thus every probability measure $P \in \G(\mS A)$ ``appears like'' a Dirac delta measure when integrating a countably affine map into $\Rinf$.
\end{proof}

\begin{thm}
The category $\SCvx_{\star}$ is isomorphic to $\M^{\G}$.
\end{thm}
\begin{proof}  The comparison functor $\mathcal{K}: \SCvx_{\star} \rightarrow \Std^{\G}$ is specified on objects by $A \mapsto (\mS A, \mS \epsilon_A)$. (In our constructions, we have shown the existence of the barycenter maps as countably affine measurable maps. The use of the notation $\mS \epsilon_A$ is used to distinguish between the measurable function ``$\mS \epsilon_A$'' viewed in $\Std_2$ versus viewing it just as the  countably affine map ``$\epsilon_A$'' in $\SCvx_{\star}$.) The  functor \mbox{$\M^{\G} \stackrel{\mathcal{W}}{\longrightarrow} \SCvx_{\star}$}  is defined on objects by  $\mathcal{W}\big( (X,h) \big) \mapsto X_h$ where $X_h$ is the super convex space consisting of the  underlying set of $X$ with the super convex space structure defined by 
$\sum_{i \in \mathbb{N}} p_i x_i = h( \sum_{i\in \bN} p_i \delta_{x_i})$ is the inverse to the comparison functor $\mathcal{K}$.  If $f: (X,h) \rightarrow (Y,k)$ is a morphism of two $\G$-algebras then, under the super convex space structure specified on the two spaces, $X_h$ and $Y_k$, it follows that
\be \nonumber
\begin{array}{lcl}
f(\sum_{i \in \Nat} p_i x_i ) &=& f\big( h( \sum_{i \in \Nat} p_i \delta_{x_i}) \big) \\
&=& k \big( \G{f}( \sum_{i \in \Nat} p_i \delta_{x_i}) \big) \\
&=& k\big( \sum_{i \in \Nat} p_i \delta_{f(x_i)}\big) \\
&=& \sum_{i \in \Nat} p_i f(x_i)
\end{array}
\ee
and hence $f$ is a countably affine map.  Hence, on arrows, the functor $\mathcal{W}$ is defined by $\mathcal{W}(f)=f$, and it is thus obvious that $\mathcal{W}$ is functorial.
The comparison functor $\mathcal{K}$ and the functor $\mathcal{W}$  specify the isomorphism $\M^{\G} \cong \SCvx_{\star}$.

\end{proof}

 \section{Remarks}  \label{sec:final}  We comment on two  different aspects of this work: (1) Other research directly related to this article, and   (2) The advantage of representing $\Std^{\G}$ as $\SCvx_{\star}$. 
 
  \vspace{.1in}
 
 (1)By using codensity monads, useful qualitative information about the Giry monad was obtained  by  Ruben Van Belle\cite{Belle} and presented 
in  the article \textit{Probability monads as codensity monads}. Belles' research showed that we could limit our focus of attention on \emph{countability}.   
That article shows the Giry monad for $\mathbf{Meas}$ can be viewed as arising from the codensity monad of a functor $\mathbf{G}:\Set_c \rightarrow \mathbf{Meas}$, where $\Set_c$ is the category of countable sets.   It is equivalent to say that $\mathbf{G}$ is the Giry monad restricted to countable measurable spaces with the powerset $\sigma$-algebra.  He employs  the countable-dimensional simplexes in defining the functor, and it is clear that the finite-dimensional simplexes can be viewed as subspaces of $\DN$. Consequently he could have chosen his functor $\mathbf{G}: \Om  \rightarrow \mathbf{Meas}$ where $\Om$ is the category with the one object $\DN$, and with arrows  as he defined in the article, $\Delta_f: \DN \rightarrow \DN$, which is the pushforward map induced by a function $f: \bN \rightarrow \bN$.  While the pushforward maps are countably affine maps, most countably affine maps $\DN \rightarrow \DN$ are not pushforward maps. 

A comparison between the $\G$-algebras on $\mathbf{Std}$ and the $\G$-algebras on complete metric spaces, $\mathbf{CMet}$ with Lipshitz-1 functions (short maps), using the Kantorovich monad illustrate an important distinction.
 The  Kantorovich monad\cite{Kant} on $\mathbf{CMet}$ has algebras that are equivalent to barycenter maps $\G(C) \rightarrow C$ sending $P \mapsto \int_C id_C \, dP$, where  $C$ is a closed convex subset of a Banach space.  Hence the codomain of these algebras are geometric spaces, i.e., embeddable into a vector space.  This issue is the same shortcoming that Doberkat\cite{Doberkat,Doberkat2} recognized in analyzing the algebras for Polish spaces with continuous maps.  Discrete spaces (with at least two points)  have no algebras because the maps are required to be continuous.  Short maps in $\mathbf{CMet}$ force the same continuity requirement, and hence discrete spaces have no algebras.
 
 \vspace{.1in}
 
(2)  The advantage of working with $\SCvx_{\star}$ rather than $\Std^{\G}$ resides in the fact that we can answer questions about existence of algebras, and obtaining knowledge about properties of $\SCvx_{\star}$ is easier that trying to figure out such properties using $\Std^{\G}$.  
Conceptually, an advantage of using super convex spaces is that we can, with a minor adjustment in the definition of a super convex space, employ ``probability amplitudes''.   To do this it is only necessary to define
 \be \nonumber
\G{\bN} = \{  \textrm{all sequences }\mathbf{p}: \bN \rightarrow  \mathbf{D}_2  \, \textrm{ such that }  \lim_{N \rightarrow \infty} \{ \sum_{i=1}^N p_i p_i^{\star} \}= 1 \}
 \ee
where $\mathbf{D}_2 = \{r e^{\imath \theta} \in \mathbb{C} \, | r \in [0,1], \textrm{ and } \theta \in [0,2 \pi) \}$ and $p_i^{\star}$ is the complex conjugate of $p_i$. 
The three axioms of a super convex space are unchanged, and countably affine sums are defined accordingly using probability amplitudes. Thus  the super convex space structure on $\T{X}$ is given by  ``$(\sum_{i\in \bN} p_i P_i)(U) = \sum_{i\in \bN} p_i p_i^{\star} P_i(U)$ for all measurable sets $U$ in $X$'', in other words, whenever we evaluate a probability measure we use the $\ell_2$-norm in evaluating the countable affine sum, 
   which is precisely what is done in quantum mechanics. 
 That same principal can be applied to any countably affine sum.  For example, the super convex space structure of $\bN$ reads as ``$\sum_{i \in \bN} p_i \, i = \min_i \{i \, | \, p_i p_i^{\star} >0 \}$''.      Nowhere in any of our theorems or lemmas do we use any special properties arising from the unit interval.  We only use the property of ``countably affine sums'' which can be defined using either  the $\ell_2$-norm or  the $\ell_1$-norm. (I am speaking loosely here; by the ``$\ell_1$-norm''  I am referring to the conditions $\sum_{i \in \bN} p_i =1$ and $p_i\ge 0$. It is the second condition which allows us to think of the first condition as $\sum_{i \in \bN} |p_i| = 1$.)\footnote{
  If we use the $\ell_2$-norm to define countable affine sums then the one-point compactification of $\mathbb{C}$,  $\mathbb{C}_{\infty}$ (the Rieman Sphere),  is a super convex space, where as $\mathbb{C}$ itself is not a super convex space.}
   
 The big advantage of using probability amplitudes only arises with the use of a dynamic model and measurement model where cancellations can occur which never arise when we restrict ourselves to using the $\ell_1$-norm  in defining a super convex space.  The axioms of a super convex space make no preference on whether we choose the $\ell_1$-norm or the $\ell_2$-norm in defining countable affine sums.
 The importance of the tensor monoidal structure of $\SCvx$ is clear to anyone who is familiar with either quantum mechanics or quantum computation\cite{NiceBook,NiceBook2}.  Under the tensor monoidal structure the no copying rule is just the statement that the function $a \mapsto a \otimes a$ is not permitted because it is not a countably affine map.   Quantum computation is a nondeterministic approach to computation, which makes extensive use of the tensor product (rather than the cartesian product), requiring nondeterministic models which suggest the use of $\G$-algebras.   
  
 In general to model nondeterminism we need barycenter maps which specify the connection between analysis on geometric (continuous) spaces, $\G(Y)$, and  analysis on the ``underlying'' spaces $Y$, which can also be of a combinatorial/discrete nature, e.g.,  $\epsilon_{\twos}:\G(\twos) \rightarrow \twos$.  It is for this reason that the algebras for $\mathbf{Std}$ are important because the algebras, isomorphic to $\SCvx_{\star}$, permit discrete spaces, thereby allowing us to model nondeterministic problems which require discrete spaces for modeling purposes. The $\G$-algebra $\epsilon_{\twos}$ can be interpreted as an ``if then else'' conditional 
which is useful  for modeling automata because given a program input $X \rightarrow \G(Y)$  we can compute the input-output map $X \rightarrow \G(Y) \rightarrow Y$ where the second arrow is a $\G$-algebra.

\end{document}